\documentclass{article}
 \usepackage[utf8]{inputenc}
 \usepackage[T1]{fontenc}
\usepackage{graphicx} 
\usepackage{import}
\usepackage{macros}
\usepackage{subcaption}
\usepackage{setspace}
\usepackage[maxnames=8]{biblatex}
\usepackage{amsmath}
\usepackage{adjustbox}
\usepackage{amssymb}

\usepackage[noend]{algpseudocode}
\usepackage{algorithm}
\usepackage{siunitx}

\usepackage{subcaption}
\usepackage{hyperref}

\usepackage{tikz}
\usepackage{pgfplots}
\pgfplotsset{compat=1.18}

\usepackage{amsthm}

\theoremstyle{plain}
\newtheorem{thm}{Theorem}
\newtheorem{pte}{Property}
\newtheorem{crl}{Corollary}
\newtheorem{lem}{Lemma}
\theoremstyle{definition}
\newtheorem{dfn}{Definition}
\newtheorem{example}{Example}
\newtheorem{rmk}{Remark}
\newtheorem{dm}{Dominance rule}
\newtheorem{pr}{Pruning rule}

\usepackage{multirow}
\usepackage{tikz}
\usepackage{pgfplots}
\usetikzlibrary{tikzmark}
\usetikzlibrary{shapes}
\usetikzlibrary{arrows.meta}
\usepgfplotslibrary{fillbetween}

\usepackage{geometry}
\definecolor{SUN}{rgb}{0.0, 0.0, 1.0}   
\definecolor{PB}{rgb}{1.0, 0.01, 0.24} 
\definecolor{CR}{rgb}{0.3, 0.5, 0.0}  
\definecolor{CA}{rgb}{1.0, 0.4, 0.0}  
\definecolor{AH}{rgb}{0.70, 0.0, 0.0}  

\usepackage{ulem}

\title{BORWin: Exact algorithm based on a Bi-Objective Relaxation for Window-constrained problems}
\author{Christian Artigues, Pascale Bendotti, Alexandre Heintzmann, \\ Sandra U. Ngueveu, Cécile Rottner}
\begin{document}

\maketitle

\begin{abstract}

A mixed integer maximization problem involving several additional constraints defined with both a lower and an upper bound is considered. It is assumed that one of such constraints is more restrictive than the others. As it can be seen as a resource window constraint, it defines the so-called window-constrained problem. From a bi-objective perspective, a 2-phase algorithm, called BORWin, is devised. It stands for Bi-Objective Relaxation for Window-constrained problems. The first phase is generic for any window-constrained problem and provides a family of upper bounds based on a bi-objective relaxation of the additional constraints. It is shown that the latter bounds strongly relate to the Lagrangian dual bounds. The second phase is derived for a variant involving a graph structure, namely the window-constrained longest-path problem on an acyclic graph. The aim is to take advantage of the upper bounds to devise an efficient label extension algorithm.
It is shown that complementary upper bounds could be derived to further improve performance in some special cases.  A typical example is when the additional constraints have special knapsack structures. This is the case for the Hydro-Unit Commitment problem with a single plant (1-HUC).   
From numerical experiments for the 1-HUC, BORWin appears to be very efficient compared to state-of-the-art approaches. 
\end{abstract}

\section{Introduction}
\label{sec:intro}
Consider a mixed integer problem of the form:
	\begin{alignat}{3}
		(P) \quad & \underset{x} {\text{max}}
		& & cx \nonumber \\
		& \text{s.t.} \quad
		& & x \in X \\
		&&& \beta \leq \gamma x \leq \alpha &  \label{main_additional_constraint}\\
		&&& \widetilde{\beta}_k \leq  \widetilde{\gamma}^k x \leq  \widetilde{\alpha}_k & \qquad k \in \{1, ..., K\} \label{additional_constraints} 
		\end{alignat}
  where $X \subset \mathbb{R}^n$, $\alpha$, $\widetilde{\alpha}_k$, $\beta$, $\widetilde{\beta}_k$ are real numbers and $\gamma$, $\widetilde{\gamma}^k$ are vectors of $\mathbb{R}^n$.
  
  Subset $X$ is a set of admissible solutions such that the corresponding subproblem:
  	\begin{alignat}{3}
		(SP) \quad & \underset{x} {\text{max}}
		& & cx \nonumber \\
		& \text{s.t.} \quad
		& & x \in X 
		\end{alignat}
		is considered efficiently solved,   
  while \eqref{main_additional_constraint} and \eqref{additional_constraints} are additional constraints.
Constraint \eqref{main_additional_constraint} is defined as the  \textit{main additional constraint} and can be arbitrarily chosen. The rational behind is to pick among the additional constraints one that is particularly structuring for the problem.

In particular, the main additional constraint is defined with both a lower and an upper bound: if quantity $\gamma x$ is regarded as a \textit{resource}, the main additional constraint can be considered as a \textit{window constraint}.
Such a type of constraints is scarcely addressed in the literature and remains hard to handle.
For example, if set $X$ corresponds to the feasible set of a shortest path problem, and constraints \eqref{main_additional_constraint}-\eqref{additional_constraints} correspond to a window for each vertex $k$ of the underlying graph, then we obtain a Window-Constrained Shortest Path Problem (WCSPP), a generalization of the classical Resource-Constrained Shortest Path Problem (RCSPP). Here the main additional constraint can be chosen as the window constraint associated with the vertex having the most restrictive bounds. {While the RCSPP is well-solved by state-of-the-art methods \cite{pugliese2013survey}, the WCSPP has been less considered in the literature as discussed in Section \ref{sec:stateoftheart}.}

\paragraph{}
If we relax additional constraints \eqref{additional_constraints}, we obtain problem $(R)$, defined as problem $(SP)$ combined with the main additional constraint \eqref{main_additional_constraint}.


A standard technique, such as a Lagrangian relaxation, can be performed on this main additional constraint, producing a sub-problem of the form 
\begin{equation}
\label{lagrangian_subproblem}
    \max_{x \in X} c x + \mu \gamma x 
\end{equation}
where $\mu$ is the lagrangian multiplier.
Such {a} subproblem can also be viewed from a bi-objective perspective \cite{bryson1991parametric}: it maximizes a vector of two (conflicting) values, $\arcvalobj{1}{x} = cx$ and $\arcvalobj{2}{x} = \gamma x$, under constraint set $X$. 
{The bi-objective relaxation of $(R)$ is thus as follows:
\begin{alignat}{3}
		(BOR) \quad & \underset{x} {\text{max}}
		& & (cx, \gamma x) \nonumber \\
		& \text{s.t.} \quad
		& & x \in X 
\end{alignat}}

In this paper, we look at the problem from the bi-objective perspective. The aim is to propose a novel algorithm to solve problems with structure $(P)$, and in particular the acyclic longest-path, namely the Window-Constrained Longest Path Problem on an Acyclic Graph (AWCLPP). 
The algorithm is {called BORWin (Bi-Objective Relaxation for Window-constrained problems)} and  
{includes} two phases. The first phase is generic for any problem (P) and provides a family of upper bounds based on a bi-objective relaxation. It is inspired by the two-phase algorithm used in bi-objective optimization \cite{ulungu1995two,visee1998two}. The second phase relies on the graph structure of the AWCLPP to implicitly enumerate feasible solutions using the family of bounds provided by the first phase by means of a label extension algorithm.

Numerical experiments show that the {BORWin} algorithm is efficient compared to two {alternative} approaches.
The first one is a Mixed-Integer Linear Program solved by CPLEX and the second one is a standard RCSPP algorithm, as considered in the literature.
This comparison is carried out on a case study corresponding to the single-plant Hydro Unit Commitment Problem (1-HUC).
Roughly speaking the 1-HUC is to maximize the daily power generation {over a discrete time horizon} with water flowing through turbines {of a plant} from an upstream reservoir. The optimization of such turbines can be modeled as an acyclic longest-path problem. Moreover, the total volume of water must satisfy window contraints at each time period. At the last time period of the horizon, tighter volume windows are applied on some reservoirs (referred to as \textit{target volumes}), in order to efficiently manage the water stock throughout a yearly horizon. Such {a} target volume constraint can thus {be} defined as the main additional constraint.

As the plant is operated on a set of discrete points, the 1-HUC can alternatively be seen as a special case of the Nested Multiple-Choice Knapsack problem (NMCKP) \cite{armstrong1982multiple}, which is itself an NP-hard particular case of the AWCLPP. 
This  generalization of the knapsack problem features a partition $ \{ N_i$, $i \in K \}$ of the objects, such that only one object from each set $N_i$ can be selected. It features also $M$ knapsack constraints, which are nested over the sets $N_i$ of the partition, in the sense that for each $m \in \{1, ..., M\}$, the $m^{th}$ knapsack constraint involves objects in a union $S_m$ of sets $N_i$, such that $S_m$ is included in the set $S_{m+1}$ corresponding to the  ${m+1}^{th}$ knapsack constraint. 

There are strong links between the NMCKP and the 1-HUC.
Indeed, each window upper-bound of the 1-HUC can be identified as a knapsack constraint \cite{heintzmann2024handling}, and as the resource cumulates from one time period to the other, such window constraints have a nested structure. The multiple-choice aspect comes from the fact that at each time period, exactly one operating point must be chosen among a discrete set of feasible points. Similarly, the window lower-bounds can be seen as  demand constraints \cite{bendotti2018min}, \textit{i.e.}, knapsack constraints with negative coefficients and such constraints can also be identified to an NMCKP.
Hence, the second phase of the BORWin algorithm will exploit both the AWCLPP and the NMCKP structures in the problem considered.

Three sets of 1-HUC instances are considered: the first is a large set of realistic instances from the literature, the second 
is a large set of realistic EDF instances, and the last is a set of difficult instances inspired by the EDF ones. 
As the first instance set will prove very easy to solve in our context, the numerical experiments will focus on the second and the third one.

 Section \ref{sec:stateoftheart} presents a literature review on bi-objective optimization, its relationship with decomposition methods and algorithms for the RCSPP, WCSPP and their longest-path variants. In Section \ref{sec:phase1}, theoretical results and algorithms for the first phase of the {BORWin}  algorithm are described. Section \ref{sec:borwin} is devoted to the second phase of {BORWin}. {In Section \ref{sec:huc}, the 1-HUC problem is described, along with a dedicated state of the art and a mathematical programming formulation.} Finally, in Section \ref{sec:results} experimental results showing the efficiency of the {BORWin} algorithm on the 1-HUC problem are provided. Concluding remarks are drawn in Section \ref{sec:conclusion}.

\section{Literature review}
\label{sec:stateoftheart}
Since Phase I of our approach is based on a bi-objective relaxation of the problem,  bi-objective optimization basics are recalled in Section \ref{sec:BO2P}.  As stated in the introduction, decomposition techniques can be used alternatively with dedicated multiobjective reformulations. Related works are presented in Section \ref{sec:decompBO}. We mentioned that phase II of the BORWin algorithm provides an exact solution framework for the AWCLPP, which can be modeled as an acyclic WCSPP. Therefore, Section \ref{sec:WCSPP} formally presents the WCSPP and its bi-objective relaxation as well as related work. Finally Section \ref{sec:WCLPP}, focuses on the AWCLPP specificities.

\subsection{Bi-objective optimization and two-phase method}
\label{sec:BO2P}

Consider a bi-objective problem where $\arcvalobj{1}{x}$ and $\arcvalobj{2}{x}$ denote the first and second objective values of a solution $x$, to be maximized.
We first recall below the definitions of Pareto-dominance, Pareto-optimality and Pareto-supported solutions.
    Consider two solutions $x$ and $x'$.
     Solution $x$ weakly Pareto-dominates solution $x'$, denoted by $x\paretoweakdom x'$, if $\arcvalobj{1}{x}\geq\arcvalobj{1}{x'}$ and $\arcvalobj{2}{x}\geq\arcvalobj{2}{x'}$. Solution $x$ strongly Pareto-dominates $x'$, denoted by $x\paretostrongdom x'$, if it weakly Pareto-dominates $x'$ and if in addition $\arcvalobj{1}{x}>\arcvalobj{1}{x'}$ or $\arcvalobj{2}{x}>\arcvalobj{2}{x'}$. A solution $x$ is Pareto-optimal if no other solution strongly dominates it. A Pareto-optimal solution is Pareto-supported if it belongs to the convex hull of the feasible points in the objective space.

The two-phase method was first introduced in~\cite{ulungu1995two} as a procedure to solve bi-objective  combinatorial  optimization problems and more specifically in \cite{visee1998two} to solve the bi-objective knapsack problem.
This classical method is particularly efficient to deal with two or even three objectives when the corresponding problem with a single objective is easy to solve. The rationale behind this method is that  the space search is dramatically reduced once Pareto-supported solutions are defined.
In the bi-objective case, the two-phase method works as follows.
The first phase  obtains Pareto-supported solutions, \textit{i.e.}, Pareto-optimal solutions that form the convex envelope of the solutions' space.
The main idea is to solve the problem with the two objectives aggregated into a single one by convex combination.
Solving the problem with various convex combination parameters can yield to different Pareto-supported solutions.
The second phase is to use an enumeration algorithm, to obtain remaining Pareto-optimal solutions within the reduced search space defined by the Pareto-supported solutions.
The most straightforward approach is the enumeration of all Pareto-optimal solutions, which can be done with a Branch \& Bound variant \cite{ulungu1995two,visee1998two} or dedicated algorithms \cite{ehlers2015computing}.
As the number of Pareto-optimal solutions can be exponential, such algorithms can become very time-consuming.

One way to reduce the computational time is to generate only a subset of the Pareto-optimal solutions, for instance with a $K$-best solutions algorithm \cite{eppstein1998finding}.
The downside is that the set of solutions may not contain the most adequate Pareto-optimal solution with respect to the decision maker's preferences, or even may not contain any feasible solution for problem $(P)$.

\subsection{Decomposition and multi-objective perspective}
\label{sec:decompBO}

In \cite{bodur2022decomposition}, the authors consider a very general integer program featuring efficiently solved subproblems combined with a (small) set of coupling constraints. Problem $(P)$ corresponds to a particular case of such problems, as it features a single subproblem with only one main additional constraint among many additional constraints. 
To handle such a generic class of problems, the authors propose a multiobjective reformulation,  coupling constraints are seen as resource constraints, and the subproblems feasible sets as non-dominated Pareto solutions.
Then, they propose a family of solution algorithms based on successive  disjunctive linear programs solved via an MILP solver, and subproblem solutions to generate new non-dominated Pareto solutions.
This approach is very generic and enables to handle problems with multiple subproblems coupled by several constraints. Clearly the algorithms performed at each iteration  are expected to be quite time-consuming as they rely in particular on solving an MILP.


In the particular case of constraint structures such that  one of the coupling constraint is more important than the others, more dedicated approaches can be considered. Typically, in \cite{bryson1991parametric}, an algorithm generating a dual bound for a linear problem with a single side-constraint is proposed. In this case, the {considered} subproblem $(SP)$ should be formulated in such a way that no binary variables are necessary. Network problems with totally unimodular constraint matrices are {an} example of such linear problems. The algorithm relies on iterated solutions of subproblem of the form \eqref{lagrangian_subproblem}, where the dual multiplier $\mu$ is updated using reduced cost of non-basic variables from the simplex solution of subproblem $(SP)$. However, the paper does not propose to exploit the obtained lower bound in a second phase, whereas our approach relies on an efficient dichotomy from the two-phase method (see in Section~\ref{sec:borwin} for details). Contrarily to our approach
, this algorithm may feature degeneracy, and slow convergence.
Nevertheless, the algorithm in~\cite{bryson1991parametric} could be compared to the first phase of our algorithm as it provides similar bounds.  

In \cite{vidal2019separable}, the considered problem features nested lower and upper bounds constraints on integer variables, with a separable Lipschitz-continuous objective function. The constraints are a special case of the constraints of problem $(P)$, as all variable coefficients $\gamma$, $\widetilde{\gamma}^k$ are 0 or 1.  For this special case, the authors propose an exact polynomial-time decomposition algorithm based on recursion arguments.

\subsection{The Window-Constrained Shortest Path Problem}
\label{sec:WCSPP}

Let $G=(U,A)$ denote a directed graph with a set of vertices $U$ and a set of arcs $A$. Each vertex $u\in U$ is associated with a window constraint $[\underline{R}_u,\overline{R}_u]$ and each arc $a\in A$ has a value $V(a)$ and a resource consumption $R(a)$. Let $\pi=(U_\pi,A_\pi)$ denote a path traversing vertex set $U_\pi\subseteq U$ and arc set $A_\pi\subseteq A$. Let $\pi_{uv}$ denote its sub-path from vertex $u\in U_\pi$ to vertex $v\in U_\pi$. Let $R(\pi)$ denote the resource consumption of path $\pi$, while $V(\pi)$ denote its value:
with $R(\pi)=\sum_{a\in A_\pi} R(a)$ and $V(\pi)=\sum_{a\in A_\pi} V(a)$. Consider a particular origin vertex $s\in U$ and a target vertex $p \in U$. Let $\Pi_{sp}$ denote the set of paths from $s$ to $p$. A path $\pi\in\Pi_{sp}$ satisfying window constraints is such that $\underline{R}_u\leq R(\pi_{su})\leq \overline{R}_u$, $\forall u\in U_\pi$. The WCSPP consists in finding a path from $s$ to $p$ of minimum value satisfying the window constraints. 
The problem with a single window constraint on the target vertex is also known as the doubly-constrained shortest path problem \cite{ribeiro1985heuristic}. As an extension of the RCSPP, the WCSPP is NP-hard.

In the survey \cite{turner2011variants}, a state-of-the-art review of different shortest path variants is provided.
The most studied related variant is the RCSPP and its elementary version (ERCSPP) \cite{Irnich2005,barrett2008engineering,pugliese2013survey},corresponding to the WCSPP with no lower bound on the window constraint usually solved by  label setting algorithms \cite{boland2006accelerated}. A label can be assimilated to a partial path issued from $s$. 
Label setting algorithm are based on dynamic programming and use the following Pareto-dominance between two labels, for minimization.

\begin{dm}
\label{dm:dominancerule1}
If two feasible partial paths $\pi_{su}$ and $\pi'_{su}$ towards the same vertex $u$ are such that 
\begin{equation*}
V(\pi_{su})\leq V(\pi'_{su})\text{ and }R(\pi_{su})\leq R(\pi'_{su})
\end{equation*}
then $\pi'$ can be discarded as it cannot lead to a better solution than $\pi$.
\end{dm}

Two-phase Lagrangian relaxation  approaches have also been considered to solve the RCSPP \cite{santos2007improved,carlyle2008lagrangian}. Focusing on a single window constraint only on vertex $p$, the first phase computes the dual Lagrangian bound $\max_{\lambda\geq 0} \min_{\pi\in\Pi_{sp}}  V(\pi) + \lambda R(\pi) - \lambda \overline{\pi}_p$. The second phase aims at closing the gap using this bound for pruning and/or guiding the search in an enumeration procedure, that could be either a branch-and-bound algorithm \cite{carlyle2008lagrangian}, a $K-$shortest path algorithm \cite{santos2007improved} or a label setting algorithm \cite{dumitrescu2003improved}.

Consider now the bi-objective relaxation of the WCSPP, obtained by ignoring all window constraints except the one on $p$. The latter constraint is transformed into an objective, thus yielding the bi-objective shortest path problem (BOSPP) $\min \{V(\pi),R(\pi)\}$, s.t. $\pi\in\Pi_{sp}$.
The link between the RCSPP and the BOSPP is often explicit in the aforementioned two-phase Lagrangian relaxation-based methods. The Lagrangian optimization results in two successive non-dominated paths $\pi^1$ and $\pi^2$ in the convex hull of the bi-objective solution space. Path $\pi^1$ such that $R(\pi^1) \leq \overline{R}_p$ provides an UB and path $\pi^2$ such that $R(\pi^2)> \overline{R}_p$ defines the Lagrangian LB. The second phase aims at searching the optimal path between those two to close the gap. Even though Lagrangian relaxation is not necessarily used in recent state-of-the-art approaches for the RCSPP, the   A*-based algorithm in \cite{ahmadi2024enhanced} is shown to be  one of the most competitive algorithm to date for the RCSPP with a single constraint on the terminal vertex. It is also used to solve the BOSPP.

More specifically,  little work has been reported on the RCSPP with equality constraints, or window constraints (WCSPP).
In \cite{ribeiro1985heuristic}, a heuristic using Lagrangian relaxation and BOSPP concepts is proposed to solve the WCSPP with only one window constraint at the last vertex. The heuristic builds an $s,p-$cut in the network and enumerates, for each origin vertex of an arc of the cut, the non-dominated points of the Lagrangian relaxation by means of a parametric shortest path procedure. Then, the partial solutions are extended by enumeration. In \cite{beasley1989algorithm}, the problem is solved by branch-and-bound where the Lagrangian relaxation bound is used to prune the search tree.

The three-phase algorithm described in \cite{zhu2007three}
solves the RCSPP with window constraints on acyclic graphs.
The main idea is to extend the graph, such that if multiple paths lead to the same vertex from the source, a new vertex is created for each of these paths.
Once the graph has been extended in this way, the problem is solved with a pseudo-polynomial time algorithm. 
In \cite{jin2024two}, a two-phase method is proposed to solve the elementary doubly-constrained shortest path problem. The first phase performs various reductions on the graph by estimating minimum and maximum consumptions along any elementary path. The second phase is an adaptation of the  depth-first search with pruning strategies of \cite{lozano2013exact}. 

In \cite{van2018shortest,van2021decomposition}, the WCSPP is considered for an application to the 1-HUC, which will be addressed in Section \ref{sec:huc}. The authors propose a label extension algorithm inspired by the standard RCSPP algorithms. Because of lower bound constraints present in the WCSPP, the dominance rule \ref{dm:dominancerule1} does not hold in general. Therefore additional sufficient conditions must be added, thus seriously weakening its effectiveness. The authors remark that under a restricting condition the following dominance rule can be applied :

\begin{dm}
\label{dm:dominancerule2}
Given two feasible partial paths $\pi$ and $\pi'$ from $s$ to $u$ , let $\Gamma(u)$  be the set of vertices reachable from $u$, if $\pi$ and $\pi'$ satisfy the condition of dominance rule \ref{dm:dominancerule1} and $\pi$ satisfies the following condition.
\begin{equation*}
R(\pi)\geq \underline{R}(v), \forall v\in\Gamma(u) 
\end{equation*}
 then $\pi'$ can be discarded.\end{dm}


\subsection{The Window-Constrained Longest Path Problem on an acyclic graph}
\label{sec:WCLPP}

For the sake of consistency with the literature, we presented in the previous section the shortest path variants of the problems WCSPP and BOSPP. 
Note that the WCSPP reduces to the window-constrained longest path problem where the value is to be maximized subject to the graph being acyclic. 
In the sequel we will consider the Acyclic Window-Constrained Longest Path Problem 
 (AWCLPP) as it directly applies to the 1-HUC problem considered as case study.
  When graph $G$, defined in section \ref{sec:WCSPP}, is acyclic, solving the WCSPP on graph $G$ is equivalent to solving the AWCLPP on the same graph $G$, with opposite arc values $-V(a)$, $\forall a\in A$.
 Even when $G$ is acyclic, the problem remains NP-hard. Likewise the Acyclic Bi-Objective Longest Path Problem (ABOLPP) can be defined.

 \begin{example}
\label{ex:WCLPP}
Figure \ref{ex:WCLPP} displays an AWCLPP instance with 5 vertices. Each vertex has a window constraint except vertex $s$. Values $(V(a),R(a))$ appear near each arc $a$. 
\begin{figure}
    \centering
    \includegraphics{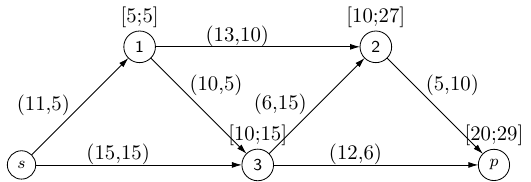}
    \caption{Graph for an instance of the AWCLPP}
    \label{fig:WCLPP}
\end{figure}
Table \ref{tab:paths} gives, the value of each path of this graph and the resource consumption for each node of each path. Each consumption exceeding a window upper bound (resp. lower bound) is marked with $\uparrow$ (resp. $\downarrow$). Here only $\pi_2$ and $\pi_3$ are feasible and the longest feasible path is $\pi_3$.

\end{example}

\begin{table}
\begin{center}
\begin{tabular}{|l|lllll|c|}
\hline
\multicolumn{6}{|c|}{path}  & value\\
\hline
\multirow{2}{*}{$\pi_1$} & $s$ & 1 & 3 & 2 & $p$ & \multirow{2}{*}{32}\\
 & 0 & 5 & 10 & 25 & 35$\uparrow$ & \\
\hline
\multirow{2}{*}{$\pi_2$}  & $s$ & 3 & $p$ & & & \multirow{2}{*}{27} \\
 & 0 & 15 & 21 & & & \\
\hline
\multirow{2}{*}{$\pi_3$}  & $s$ & 1 & 2 & $p$ & & \multirow{2}{*}{31} \\
& 0 & 5 & 15 & 25 &  & \\
\hline
\multirow{2}{*}{$\pi_4$}  & $s$ & 1 & 3 & $p$ & & \multirow{2}{*}{33} \\
 & 0 & 5 & 10 & 16$\downarrow$ & & \\
 \hline 
\multirow{2}{*}{$\pi_5$}  & $s$ & 3 & 2 & $p$ & & \multirow{2}{*}{26} \\
 & 0 & 15 & 30$\uparrow$ & 40$\uparrow$ & & \\
 \hline
\end{tabular}
\end{center}
\caption{Resource consumption and values of the paths in the graph displayed in Figure~\ref{fig:WCLPP}}
\label{tab:paths}
\end{table}


Both the WCSPP and the AWCLPP fall into the general scheme of problem $(P)$ presented in section \ref{sec:intro}. The feasible set $X$ is the set of paths in $\Pi_{sp}$ and problem (SP) is the easy-to-solve shortest/longest path problem on an acyclic graph. The main additional window constraint \eqref{main_additional_constraint} is on terminal vertex $p$, and the remaining additional window constraints on other vertices. Without loss of generality, the main additional window constraint could be defined on any vertex of the graph, not necessarily the terminal one.

Note that a single window constraint on the terminal vertex is considered in all the references cited in Section~\ref{sec:WCSPP} but \cite{van2018shortest,van2021decomposition}. In this paper we look for the extended case where window constraints occur on several vertices. Even if the first stage of the proposed approach computes also the lagrangian lower bound by ignoring all window constraints except the one on the last vertex, it is obtained  for general problem $(P)$ by an original bi-objective relaxation algorithm. Then,  the second stage uses a label extension algorithm as in \cite{van2018shortest,van2021decomposition}, except that we guide the search and perform additional pruning thanks to the lagrangian bound and further enhancements.


\section{Phase I of {BORWin}: bi-objective relaxation upper bounds}
\label{sec:phase1}

{In this section, we present the first phase relative to BORWin. The first phase  computes a family of upper bounds for $(P)$ based on the considered bi-objective relaxation {$(BOR)$}.
This bi-objective relaxation is presented in details in Section \ref{sec:borel}.
The properties yielding the family of upper bounds and the tightest of them are established in Section \ref{familyub}. The first-phase algorithm is presented in Section \ref{sec:exfirstphase}. An illustrative application of the first-phase algorithm to the AWCLPP is presented in  \ref{sec:phase1AWCLPP}.}

\subsection{Bi-objective relaxation}
\label{sec:borel}

The idea of the bi-objective relaxation is the following:  additional constraints \eqref{additional_constraints} are relaxed, and the main additional constraint \eqref{main_additional_constraint} is turned into an optimization criterion.

To do so, consider an optimal solution $x^*_{SP}$ of subproblem $(SP)$.
We distinguish 2 cases :
\begin{itemize}
    \item ``Locally Infeasible by Default" (LID), if $\gamma x^*_{SP} \leq \beta $ 
    \item ``Locally Infeasible by Excess" (LIE), if $\gamma x^*_{SP} > \alpha $
\end{itemize}
{The term ``locally'' means that feasibility is only related to $(SP)$ plus constraint \eqref{main_additional_constraint}. Side-constraints \eqref{additional_constraints} are ignored 
in this phase and will be dealt with in the second phase of the method.}

Without loss of generality, we suppose in the following that we are in the LID case.
{Indeed, if we were in the LIE case, we could rewrite constraint \eqref{main_additional_constraint} as $-\alpha \leq -\gamma x \leq -\beta$ and obtain a new formulation of problem $(P)$ falling into the LID case.}
Supposing the LID case, in this section we consider the following relaxation of $(P)$:
\begin{alignat}{3}
     (R) \quad & \underset{x} {\text{max}}
    & & cx \nonumber \\
    & \text{s. t.} \quad
    & & x \in X \nonumber \\
    &&& \beta \leq \gamma x \label{borne_min_ressource}
\end{alignat}

The results described in this section will rely on the bi-objective relaxation $(BOR)$ of $(R)$, maximizing the vector of objective $cx$ and $\gamma x$.

Each objective component is denoted by $\arcvalobj{1}{x} = cx$ and $\arcvalobj{2}{x} = \gamma x$. \vspace{2mm}
\begin{dfn}[Aggregated value]
Consider a {scalar} $\delta \in \mathbb{R}^+ $. 
For a solution $x \in X$, we define $\agg{\delta}{x}=  \arcvalobj{1}{x}+\delta \cdot \arcvalobj{2}{x}$ as the aggregated value of a solution $x$ considering {weight $\delta$ for objective $\arcvalobj{2}{x}$. {By convention when $\delta=+\infty$, $\agg{+\infty}{x}=  \arcvalobj{2}{x}$}}.
\end{dfn}
{Note that when $\delta=0$ then $\agg{0}{x}=  \arcvalobj{1}{x}$}.
Note also that a solution $x$ is Pareto-supported if there exists  $\delta$ such that $\agg{\delta}{x}\geq \agg{\delta}{x'}$ for any other solution $x'$, because in this case $x$ belongs to the convex hull of the feasible points in the objective space.

We provide 
{the following definition 
for the Pareto-equality and Pareto-difference.}

\begin{dfn}[Pareto-equality and Pareto-difference]
    Consider two solutions $x$ and $x'$.
    Solutions $x$ and $x'$ are Pareto-equal, denoted by $x\paretoeq{} x'$, if $\arcvalobj{1}{x}=\arcvalobj{1}{x'}$ and $\arcvalobj{2}{x}=\arcvalobj{2}{x'}$.
    Otherwise, $x\paretoneq{} x'$. \label{def:Pareto}
\end{dfn}

In the following, we denote by $S$ the ordered set of Pareto-supported solutions in increasing value of the first objective $\arcvalobj{1}{x} = cx$.
We do not consider in $S$ two solutions with the exact same values, meaning that for two solution $x$ and $x'$ in $S$, $x\paretoneq{} x'$.

\subsection{Family of upper-bounds}
\label{familyub}

In this section, we first define some properties for the solutions of $(BOR)$.
Besides, we show that one can obtain a family of upper bounds for $(P)$ using solutions of $(BOR)$ with a well-chosen $\delta$ for their aggregated value. 
In Section \ref{sec:exfirstphase}, we will use such results to efficiently compute these bounds in the first phase of the {BORWin} algorithm. The second phase will also exploit these results to restrict the search space.

First, we can introduce an upper bound for $\arcvalobj{1}{x^{*}}$ with $x^{*}$ the optimal solution of $(P)$.

\vspace{2mm}
\begin{pte}
    \label{pte:graphs_twophases:trivial_bound}
    Let $x$ be the Pareto-supported solution maximizing $\arcvalobj{1}{x}$ and $x^{*}$ the optimal solution of $(P)$.
    Then, $\arcvalobj{1}{x}$ is a valid upper bound for $\arcvalobj{1}{x^{*}}$.
\end{pte}
\begin{proof}
    As $x^{*}$ is a valid solution for $(P)$, it is also a valid solution for $(BOR)$.
    Consequently, if $\arcvalobj{1}{x^{*}}>\arcvalobj{1}{x}$, then $x$ cannot be the Pareto-supported solution maximizing $\arcvalobj{1}{x}$.
    We then deduce $\arcvalobj{1}{x^{*}} \leq \arcvalobj{1}{x}$. 
\end{proof}
However, it is possible to provide a tighter upper bound.
We first introduce the following lemma.

\vspace{2mm}
\begin{lem}
    \label{lem:graphs_twophases:not_pareto}
    Let $x_{1}$, $x_{2}$ and $x_{3}$ be three solutions of $(BOR)$ with $\arcvalobj{1}{x_{1}}\leq \arcvalobj{1}{x_{2}}\leq \arcvalobj{1}{x_{3}}$.
    Consider that these solutions do not dominate each other.
    Let $\delta$ be such that $\agg{\delta}{x_{1}}=\agg{\delta}{x_{3}}$.
    If $\agg{\delta}{x_{2}}<\agg{\delta}{x_{1}}$ then $x_{2}$ cannot be Pareto-supported.
\end{lem}
\begin{proof}
   { Solutions $x_1$, $x_2$ and $x_3$ do not dominate each other, therefore :  $\arcvalobj{2}{x_{1}} \geq \arcvalobj{2}{x_{2}} \geq \arcvalobj{2}{x_{3}}$.
    Thus, if we suppose $\agg{\delta}{x_{2}}<\agg{\delta}{x_{1}}$, then for any $\delta'>\delta$, $\agg{\delta'}{x_{2}}<\agg{\delta'}{x_{1}}$, and as $\agg{\delta}{x_{1}}=\agg{\delta}{x_{3}}$, for any $\delta'<\delta$, $\agg{\delta'}{x_{2}}<\agg{\delta'}{x_{3}}$. Therefore, $x_2$ is not Pareto-supported.}
\end{proof}

We introduce \textbf{Properties \ref{pte:graphs_twophases:valid_upper_bound1}} to \textbf{\ref{pte:graphs_twophases:best_R}} and \textbf{Theorem \ref{thm:graphs_twophases:best_upper_bound}} to provide a tighter upper bound for $\arcvalobj{1}{x^{*}}$.

\vspace{2mm}
\begin{pte}
    \label{pte:graphs_twophases:valid_upper_bound1}
    Let $(x_{1},x_{2})$ be a pair of consecutive solutions in $S$.
    Let $\delta$ be such that $\agg{\delta}{x_{1}}=\agg{\delta}{x_{2}}$.
    For each  solution $x$ of $(BOR)$, $\agg{\delta}{x} \leq \agg{\delta}{x_{1}}$.
    
\end{pte}

The proof of \textbf{Property \ref{pte:graphs_twophases:valid_upper_bound1}} is in \textbf{Appendix \ref{anx:graphs_twophases:valid_upper_bound1}}.

\vspace{2mm}
\begin{crl}
    \label{crl:graphs_twophases:valid_upper_bound1}
    Let $(x_{1},x_{2})$ be a pair of solutions in $S$.
    Let $\delta$ be such that $\agg{\delta}{x_{1}}=\agg{\delta}{x_{2}}$.
    Let $x_{3}$ be the solution maximizing $\agg{\delta}{x_{3}}$.
    If solution $x_{3}\paretoneq{} x_{1}$ and $x_{3}\paretoneq{} x_{2}$, then $\arcvalobj{1}{x_{3}} \in ]\arcvalobj{1}{x_{1}};\arcvalobj{1}{x_{2}}[$ and $\arcvalobj{2}{x_{3}} \in ]\arcvalobj{2}{x_{2}};\arcvalobj{2}{x_{1}}[$.
    If solution $x_{3}\paretoeq{} x_{1}$ or $x_{3}\paretoeq{} x_{2}$, then $x_{1}$ and $x_{2}$ are consecutive solutions in S.
\end{crl}
\begin{proof}
    The proof of \textbf{Property \ref{pte:graphs_twophases:valid_upper_bound1}} indicates that if $x_{3}\paretoneq{} x_{1}$ and $x_{3}\paretoneq{} x_{2}$, then $\arcvalobj{1}{x_{3}} \in ]\arcvalobj{1}{x_{1}};\arcvalobj{1}{x_{2}}[$ and $\arcvalobj{2}{x_{3}} \in ]\arcvalobj{2}{x_{2}};\arcvalobj{2}{x_{1}}[$.
    Otherwise, either $x_{1}$ or $x_{2}$ cannot be Pareto-supported, which yields a contradiction.

    In the case $x_{3}\paretoeq{} x_{1}$ or $x_{3}\paretoeq{} x_{2}$, then clearly there cannot be any Pareto-supported solution $x_{4}$ with $\arcvalobj{1}{x_{4}} \in ]\arcvalobj{1}{x_{1}};\arcvalobj{1}{x_{2}}[$ as it cannot be part of the solutions' convex hull in the objective space.
\end{proof}
For a solution $x$ of $(BOR)$, provided we know the value $\arcvalobj{2}{x}$,  Property~\ref{pte:graphs_twophases:valid_upper_bound1} gives an upper bound {of  $\arcvalobj{1}{x}$ for each pair of consecutive solutions in $S$.}
As the optimal solution $x^{*}$ of $(P)$ is also a feasible solution of $(BOR)$, it also applies to $x^{*}$.
As we do not know the value $\arcvalobj{2}{x^{*}}$, we cannot use this property directly to bound $\arcvalobj{1}{x^{*}}$.
However, we know that $\arcvalobj{2}{x^{*}}\in [\beta;\alpha]$ from the main additional constraint~\eqref{main_additional_constraint}.

\vspace{2mm}
\begin{pte}
    \label{pte:graphs_twophases:valid_upper_bound2}
    Let $(x_{1},x_{2})$ be a pair of consecutive solutions in $S$.
    Let $\delta$ be such that $\agg{\delta}{x_{1}}=\agg{\delta}{x_{2}}$.
    Let {$\bsupvalobj{1}{x_{1}}{\delta}={\agg{\delta}{x_{1}}-\delta \cdot \beta}$}.
    Then $\bsupvalobj{1}{x_{1}}{\delta}$ is an upper bound for the value of the optimal solution $\arcvalobj{1}{x^{*}}$ of $(P)$.
\end{pte}
\begin{proof}
    As $x^{*}$ is a feasible solution of $(P)$, it is also a feasible solution of $(BOR)$.
    From \textbf{Property \ref{pte:graphs_twophases:valid_upper_bound1}}, we know that {$\arcvalobj{1}{x^{*}}\leq {\agg{\delta}{x_{1}}-\delta \cdot \arcvalobj{2}{x^{*}}}$}.
    As $x^{*}$ is a feasible solution of $(P)$, then $\arcvalobj{2}{x^{*}}\in [\beta;\alpha]$, meaning that $\arcvalobj{2}{x^{*}}\geq \beta$.
    Hence, we can deduce {$\arcvalobj{1}{x^{*}}\leq {\agg{\delta}{x_{1}}-\delta \cdot \arcvalobj{2}{x^{*}}}\leq {\agg{\delta}{x_{1}}-\delta \cdot \beta}$}.
\end{proof}

{Note that in the objective space $(V_1, V_2)$, the bound $\bsupvalobj{1}{x_{1}}{\delta}$ corresponds to the value on the first objective of the point located at the intersection of  line $V_2 = \beta$ and the line going through $x_{1}$ and $x_2$.}

\textbf{Property \ref{pte:graphs_twophases:valid_upper_bound2}} provides an upper bound that can be computed independently from the values $\arcvalobj{1}{x^{*}}$ and $\arcvalobj{2}{x^{*}}$.

\vspace{2mm}
\begin{pte}
    \label{pte:graphs_twophases:best_R}
    Let $(x_{1},x_{2})$ be a pair of consecutive solutions in $S$.
    Let $\delta$ be such that $\agg{\delta}{x_{1}}=\agg{\delta}{x_{2}}$.
    For any $R\leq\beta$, {$\bsupvalobj{1}{x_{1}}{\delta}\leq {\agg{\delta}{x_{1}}-\delta\cdot R}$}.
    For any $R>\beta$, ${\agg{\delta}{x_{1}}-\delta \cdot R}$ is not a valid upper bound for the value of the optimal solution $\arcvalobj{1}{x^{*}}$ of $(P)$
\end{pte}
\begin{proof}
    Clearly, for any $R\leq\beta$, {${\agg{\delta}{x_{1}}-\delta\cdot \beta}\leq {\agg{\delta}{x_{1}}-\delta \cdot R}$}.
    Consider the case $R>\beta$.
    As $\arcvalobj{2}{x^{*}}\geq \beta$, there is no guarantee that $R \leq \arcvalobj{2}{x^{*}}$.
    Hence, it is not possible to deduce $\arcvalobj{1}{x^{*}}\leq {\agg{\delta}{x_{1}}-\delta\cdot R}$ as done in \textbf{Property \ref{pte:graphs_twophases:valid_upper_bound2}}.
    The following \textbf{Example \ref{ex:graphs_twophases:counter_example}} gives a counter-example to complete this proof in the case $R>\beta$.
\end{proof}

Property \ref{pte:graphs_twophases:best_R} shows that {for any choice of $R\neq \beta$}, the bound ${\agg{\delta}{x_{1}}-\delta\cdot R}$ would either  be invalid or looser than $\bsupvalobj{1}{x_{1}}{\delta}$.

\vspace{2mm}
\begin{example}
    \label{ex:graphs_twophases:counter_example}
    
    Consider an instance of problem $(P)$ with no additional constraints besides the main one, with bounds $\beta=3$ and $\alpha =5$.
    Consider this problem to have three distinct solutions denoted $x_{1}$, $x_{2}$ and $x_{3}$ with the following values: $(7,1), (5-\epsilon,3)$ where $5 - \epsilon < 4$ and, $(2,6)$, respectively. 
    Note that we are in the case LID.
    Indeed, the optimal solution $x_{SP}^{*}$ of the sub-problem $(SP)$ is $x_{1}$, as $\arcvalobj{1}{x_{1}}>\arcvalobj{1}{x_{2}}>\arcvalobj{1}{x_{3}}$ and $\arcvalobj{2}{x_{1}}=1<\beta=3$.
    
    Consider now the bi-objective relaxation problem $(BOR)$ associated with $(P)$.
    All the following explanations are illustrated in Figure \ref{fig:counter_example}.
    
    Clearly, solutions $x_{1}$ and $x_{3}$ are Pareto-supported, but not $x_{2}$.
    Hence, the list of Pareto-supported solutions $S$ is $\{x_{3}, x_{1}\}$.
    Due to bounds $\alpha$ and $\beta$, only $x_2$ is a feasible solution of $(P)$, meaning $x_2$ is the optimal one.

    Consider  $\delta=1$, so that $\agg{\delta}{x_{1}}=\agg{\delta}{x_{3}}$.
    Let $R$ be any number such that $R>\beta$.
    For any value $R$, ${\agg{\delta}{x_{1}}-\delta\cdot R}={1 \cdot 7 + 1 \cdot 1 - 1 \cdot R} = 8-R$.
    As $R>\beta=3$, we deduce $8-R<5$.
    For this instance, one can select $\epsilon< R-5$ such that $\arcvalobj{1}{x_{2}} = 5-\epsilon>8 - R$.
    Consequently, if $R>\beta$, one cannot deduce a valid upper bound for the value $\arcvalobj{1}{x^{*}}$ of the optimal solution of $(P)$.
\end{example}

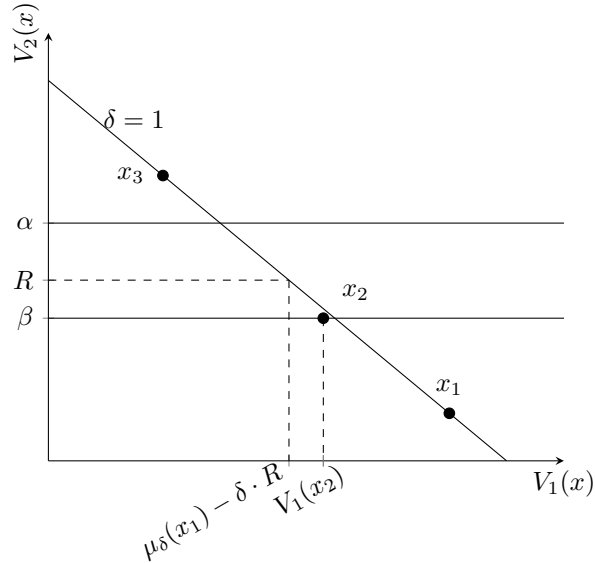
\begin{figure}[htb!]
    \centering
    
\begin{tikzpicture}
\begin{axis}[
axis lines = left,
xlabel = $V_{1}(x)$,
ylabel = $V_{2}(x)$,
x label style={at={(axis description cs:1,0)},anchor=north},
y label style={at={(axis description cs:0,1)}},
ytick={3, 3.8, 5},
yticklabels={$\beta$, $R$, $\alpha$},
x tick label style={rotate=30, anchor=east},
xtick={4.2, 4.8},
xticklabels={{{${\agg{\delta}{x_{1}}-\delta\cdot R}$}},},
extra x ticks={4.8},
extra x tick labels = {{$V_{1}(x_{2})$},},
extra x tick style={
    xticklabel style={xshift=3mm, yshift=-2mm}
    },
xmin=0,
ymin=0,
xmax=9,
ymax=9
]

\addplot[
    only marks
    ]
    coordinates {
    (7, 1) (4.8, 3) (2, 6)
    };
\node[label=above:{$x_{1}$}] at (axis cs:7, 1) {};
\node[label=above right:{$x_{2}$}] at (axis cs:4.8, 3) {};
\node[label=left:{$x_{3}$}] at (axis cs:2, 6) {};

\addplot[] coordinates {(0,8) (8,0)} node[right, pos=0.1] {$\delta=1$};
\draw [] (axis cs:0,3) -- (axis cs:10,3);
\draw [] (axis cs:0,5) -- (axis cs:10,5);

\draw[dashed] (axis cs:0,3.8) -- (axis cs:4.2, 3.8);
\draw[dashed] (axis cs:4.2,3.8) -- (axis cs:4.2, 0);
\draw[dashed] (axis cs:4.8,3) -- (axis cs:4.8, 0);

\end{axis}
\end{tikzpicture}
    \caption{Illustration of Example \ref{ex:graphs_twophases:counter_example}}
    \label{fig:counter_example}
\end{figure}

From the previous \textbf{Properties \ref{pte:graphs_twophases:valid_upper_bound2}} and \textbf{\ref{pte:graphs_twophases:best_R}}, we know that for a given pair $(x_{1},x_{2})$ of consecutive solutions in $S$, there is no $R\neq \beta$ providing a tighter upper bound  than $\bsupvalobj{1}{x_{1}}{\delta}$.
Recall that in general, the number of Pareto-supported solutions is exponential, hence there exists an exponential number of bounds.
We prove in the following that it is possible to characterize the pair $(x_{1},x_{2})$ of consecutive solutions in $S$ that provides the tightest bound $\bsupvalobj{1}{x_{1}}{\delta}$.
\begin{dfn}
    A pair $(x_{1},\delta) \in S \times \mathbb{R}^+$ is a \textit{Pareto-supported pair} if there exists solution $x_2 \in S$ consecutive to $x_1$ in $S$ satisfying $\agg{\delta}{x_{1}}=\agg{\delta}{x_{2}}$.
\end{dfn}
\begin{thm}
    \label{thm:graphs_twophases:best_upper_bound}
    Let $(x_{1},x_{2})$ be a pair of consecutive solutions in $S$ such that $\arcvalobj{2}{x_{1}}\geq \beta>\arcvalobj{2}{x_{2}}$.
    Let $\delta$ be such that $\agg{\delta}{x_{1}}=\agg{\delta}{x_{2}}$.
    Among all Pareto-supported pairs, solution $x_{1}$ and value $\delta$ minimize the value $\bsupvalobj{1}{x_{1}}{\delta}$.
\end{thm}

The proof of \textbf{Theorem \ref{thm:graphs_twophases:best_upper_bound}} is in \textbf{Appendix \ref{anx:graphs_twophases:best_upper_bound}}.

\subsection{Algorithm of Phase I}
\label{sec:exfirstphase}

The aim of this first phase is to define upper bounds, thus defining a reduced search space.
More precisely, we search for a Pareto-supported pair ($x,\delta)$, that provides the tightest bound $\bsupvalobj{}{x}{\delta}$ for the value of the optimal solution $x^{*}$ of $(P)$, according to the result provided in Theorem \ref{thm:graphs_twophases:best_upper_bound}.

The first step is to compute $x_{SP}^{*}$ the optimal solution of $(SP)$. 
If $\arcvalobj{2}{x_{SP}^{*}} \in [ \beta, \alpha ]$, the algorithm stops and returns $x_{SP}^{*}$ and $\delta=0$. Note that the associated upper bound is then $V_1(x_{SP}^{*})$.
Otherwise, one needs to find
$(x_{A},x_{B})$ a pair of consecutive solutions in $S$ such that $\arcvalobj{2}{x_{A}}\geq \beta>\arcvalobj{2}{x_{B}}$.

The idea to obtain this pair of solutions is to do an enumeration of Pareto-supported solutions following the dichotomic principle.

The steps of this enumeration are illustrated in Figure \ref{fig:first_phase}, where $\delta_{i,j}$ is the value such that $\agg{\delta_{i,j}}{x_{i}} = \agg{\delta_{i,j}}{x_{j}}$, for $(i,j)=(1,2), (3,2), (4,2), (4,5)$.

We start with two Pareto-supported solutions $x_{A}$ and $x_{B}$, not necessarily successive in $S$, such that $\arcvalobj{2}{x_{A}}\geq \beta>\arcvalobj{2}{x_{B}}$ as represented in Figure \ref{fig:first_phase_a} with solutions $x_1$ and $x_2$.
Let $\delta$ be such that $\agg{\delta}{x_{A}} = \agg{\delta}{x_{B}}$. From there, we distinguish two cases.

The first case occurs when the solution $x_{C}$ obtained by maximizing $\agg{\delta}{}$ is such that $x_{C}\paretoeq{}x_{A}$ or $x_{C} \paretoeq{} x_{B}$.
Then, we know from Corollary \ref{crl:graphs_twophases:valid_upper_bound1}, that $x_{A}$ and $x_{B}$ are successive in $S$.
As such, we can return $x_{A}$, $x_B$ and $\delta$, as  $(x_{A},x_{B})$ is a pair of consecutive Pareto-supported solutions in $S$ such that $\arcvalobj{2}{x_{A}}\geq \beta>\arcvalobj{2}{x_{B}}$.

In the second case, solution $x_{C}$ maximizing $\agg{\delta}{}$ is such that $x_{C} \paretoneq{} x_{A}$ and $x_{C} \paretoneq{} x_{B}$.
Then, we know from Corollary \ref{crl:graphs_twophases:valid_upper_bound1}, that  $\arcvalobj{1}{x_{C}} \in ]\arcvalobj{1}{x_{A}};\arcvalobj{1}{x_{B}}[$ and $\arcvalobj{2}{x_{C}} \in ]\arcvalobj{2}{x_{B}};\arcvalobj{2}{x_{A}}[$.
In other words, $x_{C}$ is a Pareto-supported solution between $x_{A}$ and $x_{B}$ in $S$.
If $\arcvalobj{2}{x_{C}}\geq \beta$, then one can consider the two solutions $x_{C}$ and $x_{B}$ rather than $x_{A}$ and $x_{B}$ as shown in Figure \ref{fig:first_phase_b}.
Otherwise, one can consider the two solutions $x_{A}$ and $x_{C}$ rather than $x_{A}$ and $x_{B}$.

The idea is to use this new pair, $x_{A}$ and $x_{C}$ (resp. $x_{B}$ and $x_{C}$), and reuse the result of Corollary \ref{crl:graphs_twophases:valid_upper_bound1} either to prove both solutions to be successive in $S$, or to find a new solution $x_{4}$.
Figure $\ref{fig:first_phase_c}$ and $\ref{fig:first_phase_d}$ show an example of the solutions $x_{4}$ and $x_{5}$ thus obtained.

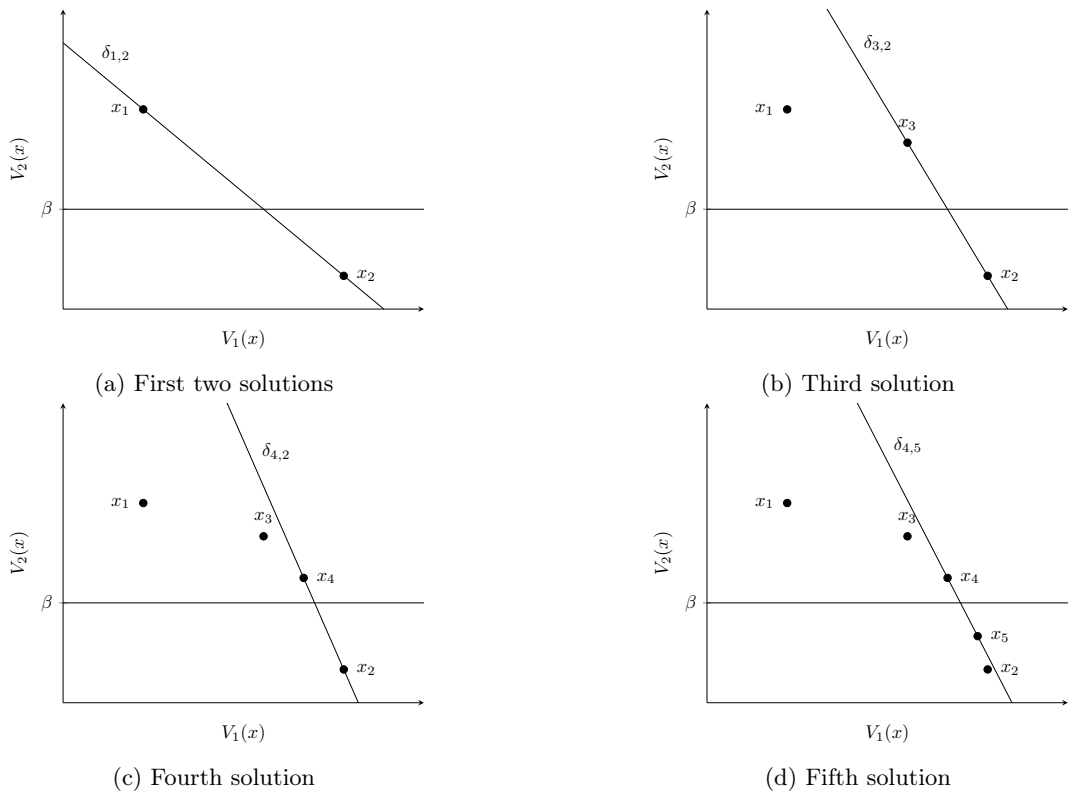
\begin{figure}[htb!]
    \centering
    \begin{subfigure}{0.48\textwidth}
    \centering
   \begin{tikzpicture}[scale=0.7]
\begin{axis}[
axis lines = left,
xlabel = $V_{1}(x)$,
ylabel = $V_{2}(x)$,
x label style={at={(axis description cs:0.5,-0.05)},anchor=north},
ytick={3},
yticklabels={$\beta$},
x tick label style={rotate=30, anchor=east},
xtick=\empty,
xmin=0,
ymin=0,
xmax=9,
ymax=9
]

\addplot[
    only marks
    ]
    coordinates {
    (7, 1) (2, 6)
    };
\node[label=left:{$x_{1}$}] at (axis cs:2, 6) {};
\node[label=right:{$x_{2}$}] at (axis cs:7, 1) {};

\addplot[] coordinates {(0,8) (8,0)} node[above right, pos=0.1] {$\delta_{1,2}$};
\draw [] (axis cs:0,3) -- (axis cs:10,3);

\end{axis}
\end{tikzpicture}
    \caption{First two solutions}
    \label{fig:first_phase_a}
    \end{subfigure}
    \begin{subfigure}{0.48\textwidth}
    \centering
    \begin{tikzpicture}[scale=0.7]
\begin{axis}[
axis lines = left,
xlabel = $V_{1}(x)$,
ylabel = $V_{2}(x)$,
x label style={at={(axis description cs:0.5,-0.05)},anchor=north},
ytick={3},
yticklabels={$\beta$},
x tick label style={rotate=30, anchor=east},
xtick=\empty,
xmin=0,
ymin=0,
xmax=9,
ymax=9
]

\addplot[
    only marks
    ]
    coordinates {
    (2, 6) (5,5) (7, 1)
    };
\node[label=left:{$x_{1}$}] at (axis cs:2, 6) {};
\node[label=right:{$x_{2}$}] at (axis cs:7, 1) {};
\node[label=above:{$x_{3}$}] at (axis cs:5, 5) {};

\addplot[] coordinates {(0,15) (7.5,0)} node[above right, pos=0.5] {$\delta_{3,2}$};
\draw [] (axis cs:0,3) -- (axis cs:10,3);

\end{axis}
\end{tikzpicture}
    \caption{Third solution}
    \label{fig:first_phase_b}
    \end{subfigure}
    
    \begin{subfigure}{0.48\textwidth}
    \centering
   \begin{tikzpicture}[scale=0.7]
\begin{axis}[
axis lines = left,
xlabel = $V_{1}(x)$,
ylabel = $V_{2}(x)$,
x label style={at={(axis description cs:0.5,-0.05)},anchor=north},
ytick={3},
yticklabels={$\beta$},
x tick label style={rotate=30, anchor=east},
xtick=\empty,
xmin=0,
ymin=0,
xmax=9,
ymax=9
]

\addplot[
    only marks
    ]
    coordinates {
    (2, 6) (5,5) (6,3.75) (7, 1)
    };
\node[label=left:{$x_{1}$}] at (axis cs:2, 6) {};
\node[label=right:{$x_{2}$}] at (axis cs:7, 1) {};
\node[label=above:{$x_{3}$}] at (axis cs:5, 5) {};
\node[label=right:{$x_{4}$}] at (axis cs:6, 3.75) {};

\addplot[] coordinates {(0,20.25) (8,-1.75)} node[above right, pos=0.6] {$\delta_{4,2}$};
\draw [] (axis cs:0,3) -- (axis cs:10,3);

\end{axis}
\end{tikzpicture}
    \caption{Fourth solution}
    \label{fig:first_phase_c}
    \end{subfigure}
    \begin{subfigure}{0.48\textwidth}
    \centering
    \begin{tikzpicture}[scale=0.7]
\begin{axis}[
axis lines = left,
xlabel = $V_{1}(x)$,
ylabel = $V_{2}(x)$,
x label style={at={(axis description cs:0.5,-0.05)},anchor=north},
ytick={3},
yticklabels={$\beta$},
x tick label style={rotate=30, anchor=east},
xtick=\empty,
xmin=0,
ymin=0,
xmax=9,
ymax=9
]

\addplot[
    only marks
    ]
    coordinates {
    (2, 6) (5,5) (6,3.75) (6.75,2) (7, 1)
    };
\node[label=left:{$x_{1}$}] at (axis cs:2, 6) {};
\node[label=right:{$x_{2}$}] at (axis cs:7, 1) {};
\node[label=above:{$x_{3}$}] at (axis cs:5, 5) {};
\node[label=right:{$x_{4}$}] at (axis cs:6, 3.75) {};
\node[label=right:{$x_{5}$}] at (axis cs:6.75, 2) {};

\addplot[] coordinates {(0,17.75) (9,-3.25)} node[above right, pos=0.5] {$\delta_{4,5}$};
\draw [] (axis cs:0,3) -- (axis cs:10,3);

\end{axis}
\end{tikzpicture}
    \caption{Fifth solution}
    \label{fig:first_phase_d}
    \end{subfigure}
    \caption{Illustrations of the dichotomic enumeration of the first phase}
    \label{fig:first_phase}
\end{figure}

Algorithm \ref{alg:graphs_twophases:first_phase} gives a pseudo-code for this first phase.
Lines 4 and 6 define the first two solutions.
Here the solutions maximizing respectively the first and the second objective are selected.
In the case where the solution maximizing the second objective has the second objective smaller than the lower bound $\beta$, then clearly no solution of $(P)$ will satisfy this bound, hence the algorithm stops (line 7 to 9).
Lines 11 to 18 are the while loop to enumerate until the new solution is Pareto-equal to the two incumbent solutions.
The updated $\delta$ value used to compute the new solution is at line 17, and the new solution is obtained at line 18.
At line 19 the algorithm returns one of the two incumbent solutions as well as the value $\delta$ used in the while loop.

\begin{algorithm}[!htb]
\caption{BORWin: first phase}\label{alg:graphs_twophases:first_phase}
\begin{algorithmic}[1]\onehalfspacing
\Require A MIP of the form of $(P)$
\State Compute $x_{SP}^{*}$ the optimal solution of $(SP)$
\If{$\arcvalobj{2}{x_{SP}^{*}}\in [\beta;\alpha]$}
    \State \Return $x_{SP}^{*}$, $\delta=0$
\EndIf
\State $x_{A}=x_{SP}^{*}$
\State $\delta \gets$ $+\infty$
\State $x_{B} \gets$ solution of $(BOR)$ maximizing $\agg{\delta}{}$
\If{$\arcvalobj{2}{x_{B}}<\beta$}
    \State No feasible solution to $(P)$
    \State \Return Infeasible
\EndIf
\State $x_{C}\gets \emptyset$
\While{$x_{C} \paretoneq{} x_{A}$ and $x_{C} \paretoneq{} x_{B}$}
    \If{$x_{C}\neq \emptyset$}
        \If{$\arcvalobj{2}{x_{C}} \geq \beta$}
            \State $x_{B}\gets x_{C}$
        \Else
            \State $x_{A} \gets x_{C}$
        \EndIf
    \EndIf
    \State $\delta\gets$$ \frac{\arcvalobj{1}{x_{A}}-\arcvalobj{1}{x_{B}}}{\arcvalobj{2}{x_{B}}-\arcvalobj{2}{x_{A}}}$ 
    \State $x_{C} \gets $ solution of $(BOR)$ maximizing $\agg{\delta}$
\EndWhile
\State \Return $x_{A}$, $x_{B}$, $\delta$
\end{algorithmic}
\end{algorithm}

With solutions $x_A$, $x_B$ and {value} $\delta$ obtained from \textbf{Algorithm \ref{alg:graphs_twophases:first_phase}}, we can define a search space for the second phase. To do so, we use four valid bounds for the optimal solution $x^{*}$ of $P$.

\begin{dfn}[Search space $\Omega$]
Search space $\Omega$ is defined as the set:
$$
\Omega = \bigg\{ x \in X \; \; \; | \; \; \; \beta \leq \arcvalobj{2}{x}\leq \alpha, \; \; \; \arcvalobj{1}{x} \leq \bsupvalobj{1}{x_A}{\delta}, \; \; \; \agg{\delta}{x} \leq \agg{\delta}{x_A} \bigg\}
$$
    
\end{dfn}

\begin{pte}
    Any optimal solution $x^*$ of $(P)$ belongs to $\Omega$.
\end{pte}
    
\begin{proof}
    Clearly, $x^{*}$ is a valid solution of $P$, hence $\beta \leq \arcvalobj{2}{x^{*}}\leq \alpha$.
Then, as shown in \textbf{Property \ref{pte:graphs_twophases:valid_upper_bound2}}, $\arcvalobj{1}{x^{*}} \leq \bsupvalobj{1}{x}{\delta}$.
Finally, $\agg{\delta}{x^{*}} \leq \agg{\delta}{x}$, otherwise there would be a contradiction with the definition of \textbf{Algorithm \ref{alg:graphs_twophases:first_phase}}.

\end{proof}

The following example illustrates how search space $\Omega$ can be derived from the result of \textbf{Algorithm \ref{alg:graphs_twophases:first_phase}}.

\begin{example}
Consider {two} instances of problem $(P)$, each with five solutions, $x_{1}, \ldots, x_{5}$.
For each instance, solution $x_{4}$ is with the highest value $V_1$.
The first instance is such that $\arcvalobj{2}{x_{4}} < \beta$, while the second instance is such that $\arcvalobj{2}{x_{4}} \in [\beta; \alpha]$.
\textbf{Figure \ref{fig:obj_space_1}} (resp. \textbf{\ref{fig:obj_space_2}})  shows search space $\Omega$ {(in the bi-objective space)} in grey, for the first (resp. second)  instance.
    
\end{example}

\begin{figure}[htb!]
    \centering
    \begin{subfigure}{0.48\textwidth}
    \centering
    \includegraphics[width=\linewidth]{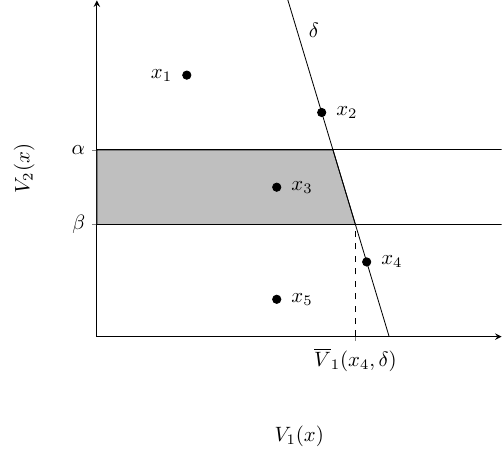}
    \caption{Instance with $\arcvalobj{2}{x_{4}} < \beta$}
    \label{fig:obj_space_1}
    \end{subfigure}
    \begin{subfigure}{0.48\textwidth}
    \centering
    \includegraphics[width=\linewidth]{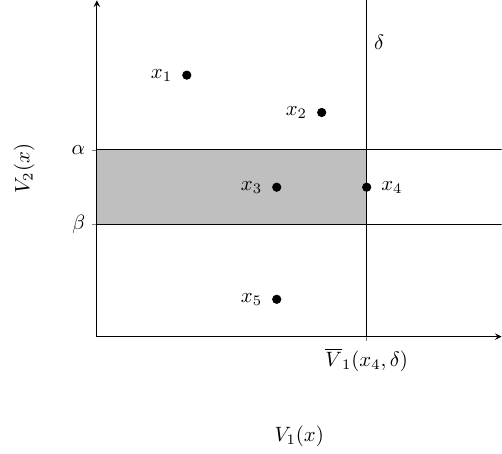}
    \caption{Instance with $\arcvalobj{2}{x_{4}} = \beta$}
    \label{fig:obj_space_2}
    \end{subfigure}

    \caption{Representation of search space $\Omega$ along with the solutions of $(BOR)$ in the objective space}

    \label{fig:obj_space}
\end{figure}

\subsection{Application to the AWCLPP}
\label{sec:phase1AWCLPP}

In this section, we present an illustrative application of Phase I using \textbf{Algorithm \ref{alg:graphs_twophases:first_phase}} to the AWCLPP.

Consider the AWCLPP instance described in \textbf{Example \ref{ex:WCLPP}}, \textbf{Figure~\ref{fig:WCLPP}} and \textbf{Table~\ref{tab:paths}}. At line 1 of \textbf{Algorithm \ref{alg:graphs_twophases:first_phase}}, the longest path according to $V_1$ (arc values $V$ in the AWCLPP) gives $\pi_4=(s,1,3,p)$ with $V_1(\pi_4)=V(\pi_4)=33$. This path is infeasible by default (LID case) since $V_2(\pi_4)=R(\pi_4)=16<\beta=20$, so $x_A=\pi_4$ (line 4). Then, at line 6, the computation of the longest path according to $V_2$ (resource $R$ in the AWCLPP) gives path $x_B=\pi_5=(s,3,2,p)$ with $V_2(\pi_5)=R(\pi_5)=40$, which is greater than $\beta$ so that no unfeasibility is detected and we enter the main loop. At line 17, we obtain $\delta=\frac{33-26}{40-16}=\frac{7}{24}$, which is illustrated in \textbf{Figure~\ref{fig:figure_phase_one_1}}.
For each arc $a$, we compute $\mu_\delta(a)=V_1(a)+\delta V_2(a)$, which gives the graph, displayed in \textbf{Figure~\ref{fig:example_awclpp_phase_1_step_1}}.

\begin{figure}
    \centering   \includegraphics[width=0.8\textwidth]{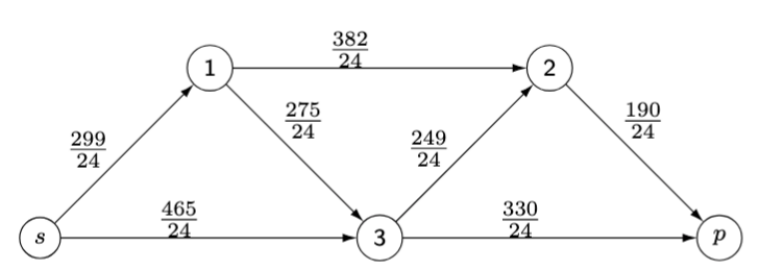}
    \caption{Graph for maximization of $\mu_\delta$ with $\delta=\frac{7}{24}$ before the first iteration}
    \label{fig:example_awclpp_phase_1_step_1}
\end{figure}

The longest path in this graph is the solution of $(BOR)$ that maximizes $\mu_\delta()$ (line 18). We obtain path $x_C=\pi_1=(s,1,3,2,p)$ which is  pareto-equal to neither $x_A$ nor $x_B$ (condition at line 11). Since $V_2(\pi_1)=35>\beta$ we set $x_B=\pi_1$ (line 14) and $\delta=\frac{33-32}{35-16}=\frac{1}{19}$, as illustrated in \textbf{Figure~\ref{fig:figure_phase_one_2}}. The solution of $(BOR)$ that maximizes $\mu_\delta()$ is computed again by computing the longest path in the graph displayed in \textbf{Figure~\ref{fig:example_awclpp_phase_1_step_2}}. There are two solutions that maximize this objective, namely $\pi_1=x_B$ and $\pi_4=x_A$ with $\mu_\delta(\pi_1)=\mu_\delta(\pi_4)=\frac{643}{19}$. Setting $x_C$ to one of these paths, the stop condition at line 11 is satisfied. The algorithm returns $x_A=\pi_4$, $x_B=\pi_1$ and $\delta=\frac{1}{19}$. We obtain $\overline{V}_1(\pi_4)=\frac{643}{19}-\frac{20}{19}\approx32.7$, which gives an upper bound of 32, improving the longest path upper bound ($V_1(\pi_4)=33$).

\begin{figure}
    \centering   \includegraphics[width=0.8\textwidth]{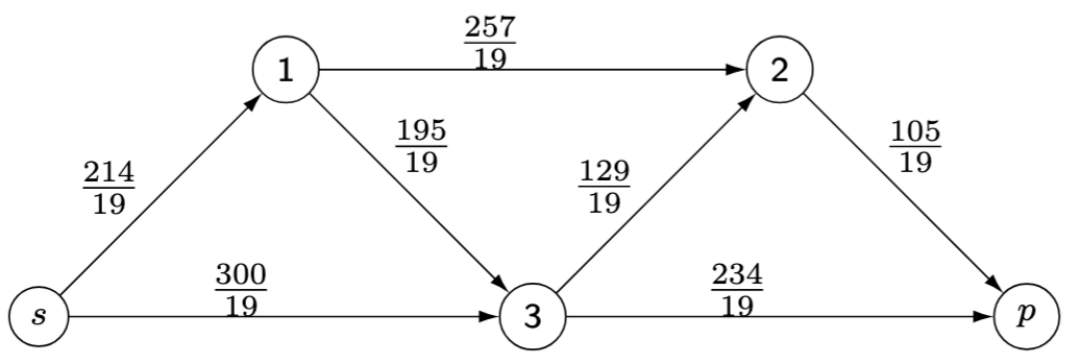}
    \caption{Graph for maximization of $\mu_\delta$ with $\delta=\frac{1}{19}$ after the first iteration}
    \label{fig:example_awclpp_phase_1_step_2}
\end{figure}

\begin{figure}[htb!]  
    \centering
    \begin{subfigure}{0.49\textwidth}
    \centering
    \includegraphics[width=\linewidth]{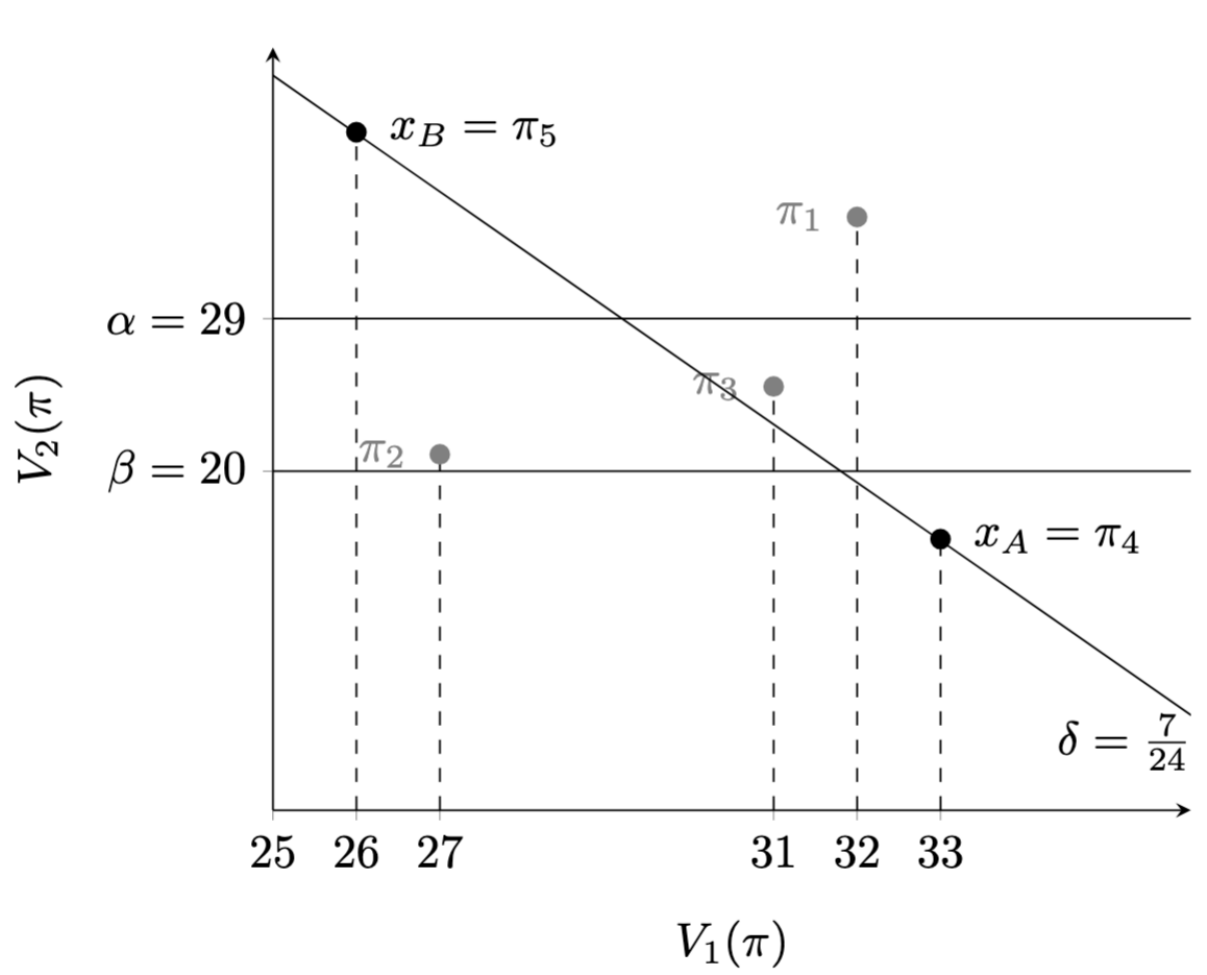}
    \caption{Initial two solutions} 
    \label{fig:figure_phase_one_1}
    \end{subfigure}
    \begin{subfigure}{0.49\textwidth}
    \centering
    \includegraphics[width=\linewidth]{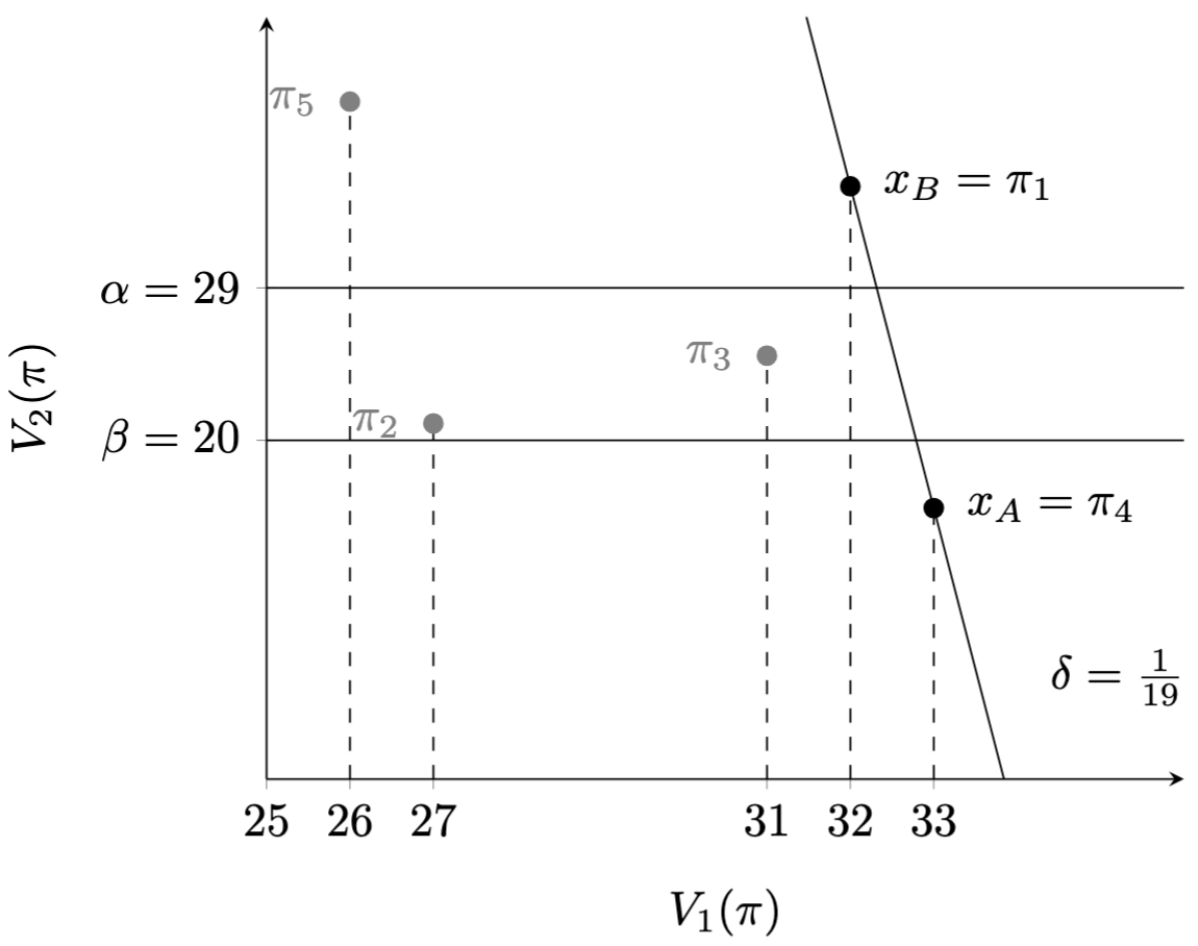}
    \caption{Two (final) solutions after the first iteration}
\label{fig:figure_phase_one_2}    \end{subfigure}
    \caption{Phase I dichotomic enumeration for AWCLPP Example~\ref{ex:WCLPP}}
\end{figure}

\subsection{Relationship to Lagrangian and Dantzig-Wolfe relaxation}

The first phase of the BORWin algorithm is described in Section \ref{sec:exfirstphase} through a bi-objective perspective, but it remains closely linked to the Lagrangian duality theory. In particular, it can be shown to return the optimal solution of the Lagrangian relaxation of $(R)$.

Indeed, suppose we dualize constraint \eqref{borne_min_ressource}, we obtain Lagrangian function $\theta$ defined for $\lambda \geq 0$:
$$
\theta(\lambda) = \max_{x \in X} cx + \lambda (\beta - \gamma x ) 
$$

With $\mbox{Ext}($X$) = \{x_1, ..., x_K\}$ the set of extreme points of $X$, we consider the Lagrangian minimization problem
    \begin{alignat}{3}
   \min_{\lambda \geq 0} \theta(\lambda) =  \quad & \underset{\sigma, \lambda} {\text{min}}
    & & \sigma + \lambda \beta \nonumber \\
    & \text{s. t.} \quad
    & & \sigma \geq (c - \lambda \gamma) x_k & & \quad \forall k \in \{1, ..., K\}\\
    &&& \lambda \geq 0
\end{alignat}
    an optimal solution $\lambda^*=\arg\min_{\mu \geq 0} \theta(\lambda) $
    and its linear programming dual, namely the Dantzig-Wolfe relaxation:
\begin{alignat}{3}
     (DW) \quad & \underset{\xi} {\text{max}}
    & & \sum_{k=1}^K (c x_k) \xi_k \nonumber \\
    & \text{s. t.} \quad
    & & \sum_{k = 1 }^K \xi_k =1 \label{convexite} \\
    &&& \beta \leq  \sum_{k=1}^K \xi_k \gamma x_k \label{ressource_dualisee} \\
    &&& \xi_k \in [0,1] & & \forall k \nonumber
\end{alignat}

First we give a practical lemma on the solution structure of $(DW)$ relying on the fact that only one constraint is dualized.

\begin{lem}
\label{lem:opt_dw}
There exists an optimal solution $\xi^*$ of $(DW)$ such that : $\exists i, j \in \{1, ..., K\}$ verifying
\begin{itemize}
    \item  $\xi^*_k = 0 \quad \forall k \not=i, j$
    \item $\xi_i + \xi_j = 1$
    \item $\xi_i (\gamma x_i) + \xi_j (\gamma x_j) = \beta$
    \item The value $v^* = \xi_i (c x_i) + \xi_j (c x_j) = \theta(\lambda^*)$ is an upper bound on the value of $(P)$
\end{itemize}
\end{lem}
\begin{proof}
    The optimal solution $\xi^*$ of this relaxation corresponds to a convex combination of extreme points of $X$ that maximizes the value $\sum_{k=1}^K (c x_k) \xi_k $ while ensuring the window lower bound $\beta$.
    
    Since there is only one dualized constraint, we can geometrically represent the solutions of problem $(R)$ in two dimensions, in the space (value $v=cx$, resource $r=\gamma x$).
    
    We consider the convex hull $H$ of the extreme points of $X$ in the 2-dimensional space $(v,r)$.
In this space, the feasible set $F$ of $(DW)$ corresponds to the intersection of $H$ with the half-space of equation ``$r \geq \beta$".

The optimal point of $(DW)$ will therefore correspond to the extreme point of $H$ that maximizes the value $v$. This point is at the intersection of the line ``$r=\beta$" (because we are in the LID case) and a facet of $H$, defined by a segment between two extreme points $x_i$ and $x_j$ of the hull $H$ (because we are in a 2-dimensional space).

\end{proof}
We denote by $\widetilde{x}_1$ and $\widetilde{x}_2$ the two extreme points with nonzero support in $\xi^*$.

In the following, we prove that the coefficient $\delta$ obtained at convergence of Phase I corresponds to an optimal solution $\lambda^*$ of the Lagrangian minimization problem.
Moreover, $\widetilde{x}_1$ and $\widetilde{x}_2$ are optimal solutions of $(BOR)$ maximizing $\mu_{\delta}$.

\begin{pte}\label{prop:equiv-LR}
Let $(x_A, \delta)$ be the result obtained at convergence of Algorithm \ref{alg:graphs_twophases:first_phase} (when not infeasible). Then, $\mu_\delta (\widetilde{x}_1) = \mu_\delta (\widetilde{x}_2) = \mu_\delta (x_A) $. 
Moreover,  $\delta \in \arg\max \theta(\lambda)$.
\end{pte}

\begin{proof}



Note that the goal of Algorithm \ref{alg:graphs_twophases:first_phase}  is precisely to find a pair of points $x_A \in X$ and $x_B \in X$ such that:
\begin{itemize}
  \item [(i)] Solutions $x_A$ and $x_B$ are Pareto-supported (thus are extreme points of the convex hull $H$)
  \item  [(ii)]Solutions $x_A$ and $x_B$ are consecutive in the ordered set of supported-Pareto points $S$ (therefore the segment between $x_A$ and $x_B$ is a facet of $H$)

  \item  [(iii)]In space (value $v$, resource $r$), consider the line  with equation $\theta = v + \lambda( \beta - r)$ passing through $x_A$ and $x_B$ ($\theta$ and $\lambda$ being determined by this condition). Solutions $x_A$ and $x_B$ are such that $\theta$ is minimal among pairs satisfying properties defined in (i) and (ii).
\end{itemize}

Thus, we can conclude that solutions  $\widetilde{x}_1$, $\widetilde{x}_2$, $x_A$ and $x_B$ are all on the same line $\theta = v + \lambda( \beta - r)$, using Lemma \ref{lem:opt_dw}.
This line has slope $-\delta$, as $\delta= \frac{\arcvalobj{1}{x_{A}}-\arcvalobj{1}{x_{B}}}{\arcvalobj{2}{x_{B}}-\arcvalobj{2}{x_{A}}}$.
Therefore, $\mu_\delta (\widetilde{x}_1) = \mu_\delta (\widetilde{x}_2) = \mu_\delta (x_A) $.

There  exists $\lambda^* \in \arg\max \theta(\lambda)$ such that the slope of this line is also equal to $-\lambda^*$: indeed, if window lower bound $\beta$ decreases by one unit, a new optimal solution of $(DW)$ can be found on the line between $\widetilde{x}_1$ and $\widetilde{x}_2$, with increased objective value corresponding to an optimal dual value of constraint \eqref{ressource_dualisee}.
\end{proof}

All in all, the upper bound $V_1(x_A, \delta)$ obtained at convergence of Algorithm 1 corresponds exactly to the Lagrangian upper bound $\min_{\lambda \geq 0} \theta(\lambda)$.

\section{Phase II of BORWin: 
a label-extension procedure for the AWCLPP}
\label{sec:borwin}

We consider now the AWCLPP as described in section \ref{sec:WCLPP}. The second phase of the BORWin algorithm aims at solving the AWCLPP using value $\delta$ and search space $\Omega$ obtained from Algorithm \ref{alg:graphs_twophases:first_phase}. The dedicated label extension algorithm solves the LPP problem with objective $\agg{\delta}{.}$. This algorithm uses  value $\delta$ both to guide the search and to prune  sub-paths, as Lagrangian relaxation-based label extension algorithms do for the RCSPP \cite{dumitrescu2003improved}. 

This section is organised as follows. In section \ref{sec:hybridpaths}, we introduce the concept of hybrid paths and establish properties on which the algorithm efficiency relies. 
We provide in section \ref{sec:UB-mutichoiceknapsack} an additional upper bounding scheme when the AWCLPP has the special structure of a multiple choice knapsack with nested window constraints (NMCKP) (see definition in the introduction).
The overall second-phase algorithm is described in section \ref{sec:secondphasealg}.

\subsection{Properties of hybrid paths and dominance rule}
\label{sec:hybridpaths}

We introduce the concept of hybrid paths that will be used to build feasible paths of the AWCLPP defined on graph $G=(U,A)$.

The notations used in Section \ref{sec:borwin} are the following.  The set of all paths  from vertex $s$ to vertex $p$ in  $G$ is denoted by $\Pi_{sp}$.
Any path $\pi$ in $\Pi_{sp}$ can be identified using its vertex and edge sets $\pi=(U_\pi,A_\pi)$.
Let $\pi_{uv}$ denote its sub-path from vertex $u\in U_\pi$ to vertex $v\in U_\pi$.
The AWCLPP aims at finding a path of maximal value from $s$ to $p$ in acyclic graph $G$. Such a path must also satisfy all window constraints, \textit{i.e.}, $\underline{R}_u \leq R(\pi_{su}) \leq \overline{R}_u$, $\forall u\in U_\pi$. We also identify $V_1(\pi)=V(\pi)=\sum_{a\in A_\pi} V(a)$ as the value of a path $\pi\in\Pi_{sp}$ and $V_2(\pi)=R(\pi)=\sum_{a\in A_\pi} R(a)$ as its used resource amount. 




\begin{dfn}[Hybrid path $\hat\pi$]
    For a given vertex $v \in U$, a \textit{hybrid path} $\hat\pi\in\Pi_{sp}$ \textit{through vertex $v$} in graph $G$ is  composed of two \textit{sub-paths}:
    $\pi$ from $s$ to $v$, and $\bar{\pi}$ from $v$ to $p$.
    Sub-path $\pi$ satisfies the  windows constraints of the AWCLPP, while $\bar\pi$ relaxes them and is the longest path from $v$ to $p$ according to $\agg{\delta}{.}$.
\end{dfn}

From such hybrid paths, it is possible to define upper bounds of value $\agg{\delta}{\pi}$ for any path $\pi$ feasible for the AWCLPP.

\vspace{2mm}
\begin{thm}\label{thm:ub:borwin2}
    Let $\hat\pi=(\pi,\bar\pi)$ be a hybrid path  through vertex $v$. 
   For any hybrid path $\hat\pi'$ through $v$, feasible for the AWCLPP, it holds that $\agg{\delta}{\hat\pi}\geq \agg{\delta}{\hat\pi'}$.
\end{thm}
\begin{proof}

    Suppose $\agg{\delta}{\hat\pi}<\agg{\delta}{\hat\pi'}$, then as $\hat\pi=(\pi,\bar\pi)$ and $\hat\pi'=(\pi,\bar\pi')$, we deduce $\agg{\delta}{\bar\pi}<\agg{\delta}{\bar\pi'}$.
    Hence, there is a contradiction, as $\bar\pi$ would not maximize $\agg{\delta}{.}$.
\end{proof}


A lower bound on the value $\agg{\delta}{\pi^{*}}$, with $\pi^{*}$ the optimal path of the AWCLPP can also be obtained.

\vspace{2mm}
\begin{thm}
    \label{thm:graphs_twophases:lower_bound}
    Let $\pi$ and $\pi^{*}$ be, respectively, a feasible and an optimal path for the AWCLPP.
    Let $\underline{\mu}_\delta(\pi)= V(\pi)+\delta \minresource{p}$.
    Then $\pi^{*}$  is such that $\agg{\delta}{\pi^{*}} \geq \underline{\mu}_\delta(\pi)$.
\end{thm}
\begin{proof}
   Clearly, $\arcvalobj{}{\pi^{*}}\geq \arcvalobj{}{\pi}$, $\arcres{\pi^{*}} \geq \minresource{\pi}$  and by construction $\delta \geq 0$.
    Hence $\agg{\delta}{\pi^{*}} \geq \underline{\mu}_\delta(\pi) =  \arcvalobj{}{\pi}+\delta \minresource{p}$.

\end{proof}

 According to Theorems \ref{thm:ub:borwin2} and \ref{thm:graphs_twophases:lower_bound}, the following pruning rule 
 can be enounced.
\begin{pr}
\label{pr:pruningrule}
Let $\pi$ be a feasible solution. For a given vertex $v$, consider a hybrid path $\hat \pi' = (\pi',\bar\pi')$ through $v$, such that $\agg{\delta}{\hat \pi'}\leq \underline{\mu}_\delta(\pi)$.
Then, sub-path $\pi'$ can be discarded.
\end{pr}

Pruning rule~\ref{pr:pruningrule} is useful for implicit enumeration algorithms such as the one proposed in Section \ref{sec:secondphasealg}.
We also propose a weakened variant of the RCSPP dominance rule \ref{dm:dominancerule1}.
\begin{dm}
\label{dm:dominancerule3}
If two feasible  sub-paths $\pi_{su}$ and $\pi'_{su}$ from $s$ to $u$ are such that 
\begin{equation*}
R(\pi_{su})= R(\pi'_{su})\text{ and }V(\pi_{su})\leq V(\pi'_{su}) 
\end{equation*}
then $\pi'_{su}$ can be discarded as it cannot lead to a better solution than $\pi_{su}$.
\end{dm}


\subsection{Complementary upper bounds based on other underlying easily solved structures}
\label{sec:UB-mutichoiceknapsack}

Problem $(P)$ has been defined as a structured problem featuring a shortest/longest path sub-structure --represented by feasible set $X$-- coupled to window constraints \eqref{main_additional_constraint} - \eqref{additional_constraints}.
Based on such a structure, phase I of the BORWin algorithm computes a Lagrangian upper bound on a weighted combination of values $V_1$ and $V_2$.

However, other useful underlying structures may exist in $(P)$, and such structures can also be exploited to compute  upper bounds complementary to the upper bounds obtained in phase I.

Consider a subproblem $(Q)$ $\max_{x \in Y} cx $ of $(P)$, \textit{i.e.}, the solution set $Y$ contains the feasible set of $X$ :
$$
\{ x \in X \; | \; \eqref{main_additional_constraint} -  \eqref{additional_constraints} \} \subset Y
$$

Suppose problem $(Q)$ is easily solved, in the sense that there exists a very efficient algorithm to obtain an optimal solution.
Then, BORWin algorithm can benefit from the upper bound on $(P)$ obtained by solving $(Q)$.

More precisely, we suppose that the structure of $(Q)$ is preserved by partial variable fixing, \textit{i.e.}, for any partial solution $\overline{x}$  - defined as a dictionary containing the following pairs (index, partial solution value for this index), problem $(Q(\overline{x}))$ remains easily solved:
	\begin{alignat*}{3}
		(Q(\overline{x})) \quad & v_Q(\overline{x}) \; = \;  \max_x cx
		& &   \\
		& \text{s.t.} \quad \;
		  x_i = \overline{x}_i & & \quad \forall i \in \overline{x}
		\end{alignat*}

Then, phase II of BORWin algorithm can efficiently rely on the following very classical result.
\begin{pte}
Consider $v$ the value of a given feasible solution to $(P)$.
For any partial solution $\overline{x}$, if $v_Q(\overline{x}) \leq v$ then partial solution $\overline{x}$ cannot lead to a solution with better value than $v$ (and therefore it can be ``pruned").
\end{pte}

Note that for such a bound to be complementary to the Lagrangian bound from phase I, it is necessary that feasible set $Y$ be different from $X$. 

Set $Y$ can contain a selection of window constraints \eqref{main_additional_constraint} -  \eqref{additional_constraints} for which efficient algorithms are known.
In the case where set $Y$ can be identified to a matroid, efficient greedy algorithms are optimal \cite{faigle2009general}.    
In another case where the union of (some of) these window constraints features a nested multiple-choice knapsack structure, efficient linear relaxation algorithms exist \cite{armstrong1982multiple,heintzmann2023efficient}. Moreover, if such a nested multiple-choice knapsack structure features symmetric weights, an efficient dedicated A* algorithm exists to obtain an integer optimal solution \cite{heintzmann2023efficient}.

Returning to the AWCLPP, 
if the graph possesses such a special structure $Y$,
let $UB(\pi)$ denote the specific upper bound associated with the partial solution defined by a sub-path $\pi$. If no special structure exists, we can always set $UB(\pi)=V(\pi)$. We then define the following additional pruning rule.

\begin{pr}\label{pr:ub}
Let $\pi$ be a feasible solution and $\pi'$ a sub-path issued from $s$. If $UB(\pi')\leq V(\pi)$, then $\pi'$ can be discarded.
\end{pr}

\subsection{Algorithm of Phase II}
\label{sec:secondphasealg}

\textbf{Algorithm \ref{alg:graphs_twophases:second_phase}} describes the enumeration of the second phase, which works as follows.
A hybrid path $\hat\pi=(\pi,\overline{\pi})$ can be viewed as a label used in standard label extension algorithms.
As soon as a new label/hybrid path $\hat\pi$ is generated, the values $V(\hat\pi_{sv})$, $R(\hat\pi_{sv})$ and $\agg{\lambda}{\hat\pi_{sv}}$ are stored in a list $L_v$ for each vertex $v\in U_{\hat\pi}$ to avoid useless recomputations.
A preprocessing phase at line \ref{Bp2:computebolp} computes the longest path $\bar\pi_{up}$ according to $\agg{\delta}{.}$ from each vertex $u\in U$ to $p$.
The list $L$ of hybrid paths is then initialized to a unique hybrid path $(\emptyset,\bar\pi_{sp})$ at line \ref{BP2:initL}. This corresponds to the initial label with an empty feasible  sub-path and the longest path in $G$ according to $\agg{\delta}{.}$.
List $L$ will be kept sorted by decreasing value $\agg{\delta}{.}$ during the whole algorithm to drive the search towards the most promising direction in terms of the Lagrangian upper bound.
The best path found so far ($\pi^*$) is initialized as the empty path at line \ref{Bp2:initbest}.
As long as $L$ is not empty, the main loop (lines \ref{Bp2:mainloop}--\ref{Bp2:endmainloop}) proceeds with the following steps.
The first hybrid path $\hat\pi=(\pi,\bar\pi)$ in $L$ is selected at line \ref{Bp2:poplabel}. Then, two cases are considered.

The first case is when $\hat\pi$ is feasible, which is tested at line \ref{Bp2:checkfeasible} by checking window constraints for all vertices in longest path $\bar\pi$. Then, lines \ref{Bp2:prunelabels}--\ref{Bp2:endprunelabels} remove all labels $\hat\pi'=(\pi',\bar\pi')$ in $L$ whose upper bounds $\agg{\delta}{\hat\pi'}$ is not larger than the lower bound $\underline{\mu}_\delta(\hat\pi)$ defined by $\hat\pi$ (pruning rule \ref{pr:pruningrule}).
We also remove label $\hat\pi'$ when $UB(\pi')\leq V(\bar\pi)$ (pruning rule \ref{pr:ub}).

The second case is when $\hat\pi$ is not feasible due to $\bar\pi$  and we consider extensions of $\pi$.
A classical label extension algorithm would then consider the last vertex $u$ in the feasible sub-path $\pi$ and its extension through each arc $(u,v)\in A$. 
Taking advantage of the longest path already evaluated $\bar\pi$, we generalize this extension scheme to generate more new hybrid paths, thus obtaining feasible paths faster. For each vertex $u$ of $\bar\pi$ such that
$\hat\pi_{su}$ is a feasible sub-path, which is checked at line \ref{Bp2:checksubpath}, all feasible extensions of $\hat\pi_{su}$ via each arc $(u,v)\in A\setminus A_{\bar\pi}$ (the successor of $u$ in $\bar\pi$ can be excluded) are considered inside the loop delimited by lines \ref{Bp2:extensionloop}--\ref{Bp2:endextensionloop}. This extension, denoted by $\pi'$, is checked according to the weak dominance rule \ref{dm:dominancerule3} against each $\pi''$ stored in list $L_v$ of feasible sub-paths that have already reached vertex $v$ (line \ref{Bp2:checkbdm}). If sub-path $\pi'$ is not dominated, new hybrid path $\hat\pi'$ is created. Then, if the upper bounds $\mu_\delta(\hat\pi')$ and $UB(\pi')$  compared to the best known solutions do not allow to discard it via pruning rules \ref{pr:pruningrule} and \ref{pr:ub}, hybrid path $\hat\pi'$ and sub-path $\pi'$ are added to $L$ and $L_v$, respectively (lines \ref{Bp2:checkbdm}--\ref{Bp2:endcheckbdm}).

\begin{algorithm}
\caption{BORWin: second phase}\label{alg:graphs_twophases:second_phase}
\begin{algorithmic}[1]\onehalfspacing
\Require An AWCLPP graph $G=(U,A)$ and a scalar $\delta$
\State Compute $\bar\pi_{up}$,  the longest $u-p$  path such that $\agg{\delta}{\bar\pi_{up}}$ is maximized, $\forall u\in U$.\label{Bp2:computebolp}
\State $L \gets [(\emptyset, \bar\pi_{sp})]$\label{BP2:initL}
\State $\pi^* \gets \emptyset$\label{Bp2:initbest}
\While{$L\neq \emptyset$} \label{Bp2:mainloop}
    \State Pop $\hat\pi=(\pi,\bar\pi)$ the first element of $L$\label{Bp2:poplabel}
    \If{$\forall v\in U_{\bar\pi}$, $R(\hat\pi_{sv})\in[\minresource{v},\maxresource{v}]$ ($\hat\pi$ is feasible)}\label{Bp2:checkfeasible}
        \For{$\hat\pi'=(\pi',\bar\pi')\in L$}\label{Bp2:prunelabels}
            \If{$\agg{\delta}{\hat\pi'}\leq \underline{\mu}_{\delta}(\hat\pi)$ or $UB( \pi')\leq V(\hat \pi)$}
                \State remove $\hat\pi'$ from $L$
            \EndIf
        \EndFor\label{Bp2:endprunelabels}
        \If{$\pi^* = \emptyset$ or $V(\hat\pi)>V(\pi^*)$}
            \State $\pi^* \gets \hat\pi$
        \EndIf
    \Else
        \For{$u\in U_{\bar\pi}\setminus\{p\}$}\label{Bp2:extensionloop}
            \State $\pi' \gets \pi\cup \bar\pi\setminus\bar\pi_{up}$
            \If{$R(\pi'_{sv})\in[\minresource{v},\maxresource{v}]$, $\forall v \in \bar\pi\setminus\bar\pi_{up}$}\label{Bp2:checksubpath}
                \For{each arc $(u,v)\in A\setminus A_{\bar\pi}$}
                    \State ${\pi'} \gets {\pi'}\cup (u,v)$
                    \If{$\forall \pi''\in L_v$, $R(\pi'')\neq R(\pi') \vee V(\pi'')>V(\pi')$}\label{Bp2:checkbdm}
                        \State $\hat\pi'\gets (\pi',\bar\pi_{vp})$
                        \If{$\agg{\delta}{\hat\pi'}> \underline{\mu}_{\delta}(\pi^*)$ and $UB(\pi')> V(\pi^*)$}
                            \State add $\hat\pi'$ to $L$ according to decreasing $\agg{\delta}{.}$
                            \State add $\pi'$ to $L_v$
                        \EndIf
                    \EndIf\label{Bp2:endcheckbdm}
                \EndFor
            \EndIf
        \EndFor\label{Bp2:endextensionloop}
    \EndIf
\EndWhile\label{Bp2:endmainloop}
\State \Return $\pi^*$
\end{algorithmic}
\end{algorithm}

\subsection{Application to the AWCLPP}
\label{sec:AWXLPP-application}
In this section, we present an illustrative application of Phase II using Algorithm 2 to the AWCLPP.

Consider again the AWCLPP instance described in \textbf{Example \ref{ex:WCLPP}}, \textbf{Figure~\ref{fig:WCLPP}} and \textbf{Table~\ref{tab:paths}}.
At the first phase, described in Section~\ref{sec:phase1AWCLPP}
, we obtained $\delta= \frac{1}{19}$.

We start the second phase with a first hybrid path maximizing $\mu_\delta()$. We choose $\hat\pi^{1}=(\emptyset, \{(s, 1), (1, 3), (3, 2), (2, p)\})$. For this example, we ignore any specific upper bound $UB(\pi^{1})$.
We initialize list $L=[\hat\pi^{1}]$.
\textbf{Figure \ref{fig:figure_phase_two_1}} gives the search space and the known solution at the start of the second phase.
The hybrid path satisfies all windows constraints except the one of $p$ and is thus infeasible.
Indeed, $R(\hat\pi^{1})=35$ while $\maxresource{p}=29$.
We then remove $\hat\pi^1$ from $L$ and create new hybrid paths from $\hat\pi^1$.
\begin{itemize}
    \item The only arc starting with $s$, different from $(s, 1)$ is $(s, 3)$.
    The longest path from $3$ to $p$ with  $\delta=0.053$ is $\{(3, p)\}$.
    Hence we have $\hat\pi^{2} = (\{(s, 3)\}, \{(3, p)\})$ with $V(\hat\pi^{2})=27$ and $R(\hat\pi^{2})=21$.
    \item The only arc starting with $1$, different from $(1, 3)$ is $(1, 2)$.
    The longest path from $2$ to $p$ with  $\delta=0.053$ is $\{(2, p)\}$.
    Hence we have $\hat\pi^{3} = (\{(s, 1), (1, 2)\}, \{(2, p)\})$ with $V(\hat\pi^{3})=31$ and $R(\hat\pi^{3})=25$.
    \item The only arc starting with $3$, different from $(3, 2)$ is $(3, p)$.
    As $p$ is the target vertex, we have $\hat\pi^{4} = (\{(s, 1), (1, 3), (3,p)\}, \emptyset)$ with $V(\hat\pi^{4})=33$ and $R(\hat\pi^{4})=16$.
    However,  sub-path $\pi^4$ of $\hat\pi^4$ does not satisfy the constraints.
    Indeed, this path uses 16 amount of resources from $s$ to $p$ while $\minresource{p}=20$.
    Consequently, we discard path $\hat\pi^4$.
    \item There are no arcs starting with $2$ different from $(2, p)$.
\end{itemize}
Two hybrid paths were generated from $\hat\pi^1$, namely $\hat\pi^2$ and $\hat\pi^3$.
By adding these hybrid paths to $L$ in decreasing values of $\agg{\delta}{.}$, we have $L= [\hat\pi^3, \hat\pi^2]$.
\textbf{Figure \ref{fig:figure_phase_two_2}} gives the search space and known solutions after the first iteration of phase II.
Hybrid path $\hat\pi^3$ is a feasible solution.
We save $\hat\pi^3$ as the current best solution found $\pi^*$ and remove $\hat\pi^3$ from $L$.
We compute the lower bound 
We compute the lower bound $\underline{\mu}_{\delta}(\pi^*))=V(\pi^*)+\delta_\cdot \minresource{p}=\frac{609}{19}$.
As $\agg{\delta}{\hat\pi^2} = \frac{534}{19} < \underline{\mu}_{\delta}(\hat\pi^3)$, we can discard solution $\hat\pi^2$.
List $L$ being now empty, the algorithm stops and returns the path corresponding to $\pi^3$ as the optimal solution.

\begin{figure}[htb!] 
    \centering
    \begin{subfigure}{0.49\textwidth}
    \centering
    \includegraphics[width=\linewidth]{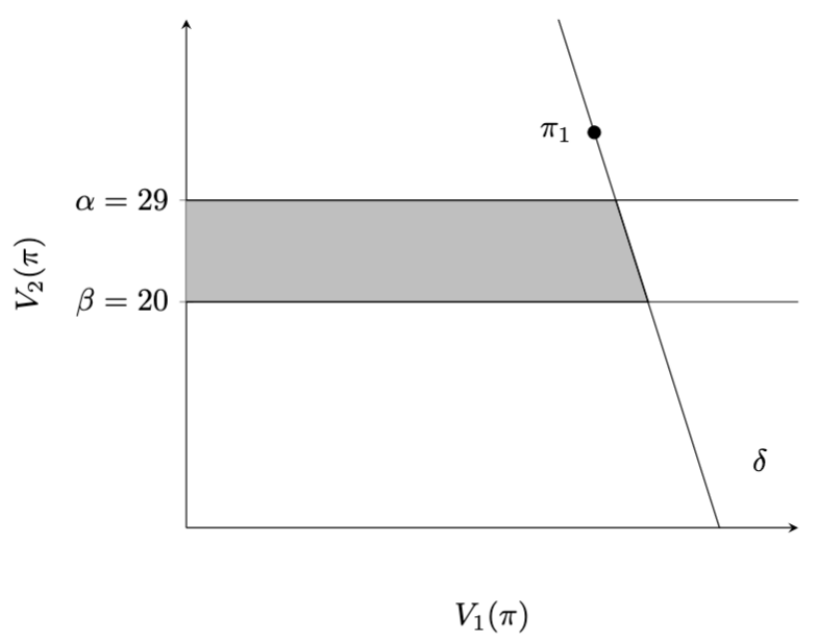}
    \caption{Initial seach space} 
    \label{fig:figure_phase_two_1}
    \end{subfigure}
    \begin{subfigure}{0.49\textwidth}
    \centering
    \includegraphics[width=\linewidth]{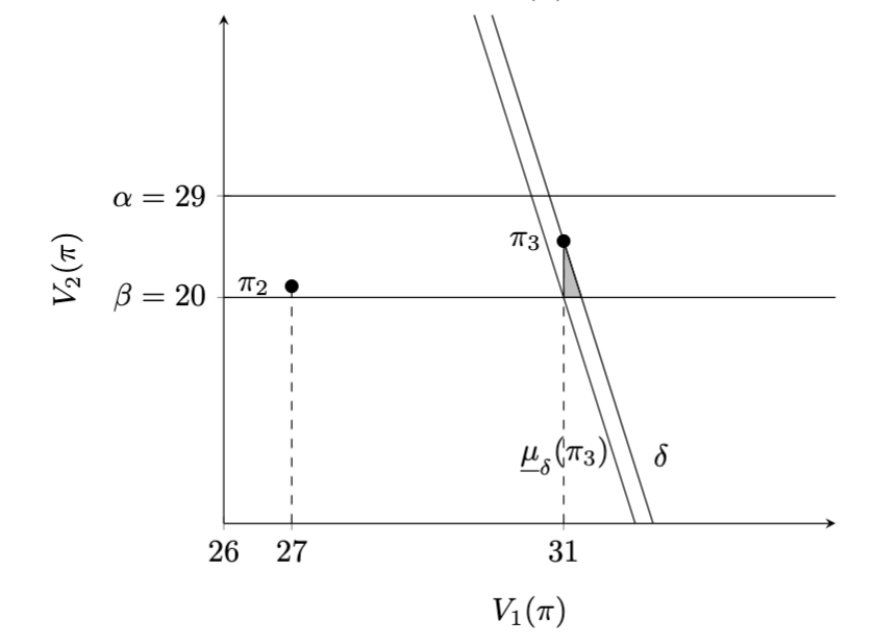}
    \caption{Search space after the first iteration}
     \label{fig:figure_phase_two_2}
    \end{subfigure}
    \caption{Illustration of the search space explored by Phase II algorithm for Example~\ref{ex:WCLPP}}
\end{figure}


Note that the proposed implementation is to store for each vertex $v$ a list  $L_v$ containing all enumerated non-dominated feasible sub-paths from $s$ to $v$.

\begin{rmk}
\label{rmk:efficient}
This implementation is particularly effective if there are many paths using the exact same amount of resource between vertex $s$ and a vertex $v$, as the lists $L_v$ at each vertex will remain small thanks to the weakened dominance rule \ref{dm:dominancerule3}.
This appears when there are symmetries in resource use. This happens to be the case for the NMCKP when some weights of the items are identical.
\end{rmk}

\section{Application to the single-plant Hydro Unit Commitment problem}
\label{sec:huc}

The Hydro Unit Commitment  (HUC) problem specific to hydroelectric plants is part of the electricity production planning problem called Unit Commitment Problem (UCP).
At Electricité de France (EDF), large-scale instances of the UCP are solved on a daily basis by a Lagrangian decomposition \cite{renaud1993daily}.
The HUC problem is the sub-problem dedicated to the valleys and is solved using an MILP formulation \cite{hechme2010short}.  

In this paper the considered use-case is the single-plant HUC (1-HUC). Recall it features window constraints as main additional constraints. The idea is to see how the BORWin algorithm compares in terms of performance with classical 1-HUC approaches in the literature. 

In this section, a description of the 1-HUC along with a mathematical formulation is proposed in Section~\ref{sec:MILP}. A state of the art relative to the 1-HUC focusing on dynamic programing is presented in Section~\ref{sec:litterature_1-huc}. It is followed in Section~\ref{sec:1-huc_graph} by  the graph representation dedicated to model the 1-HUC as an AWCLPP. 

\subsection{Description and formulation for the 1-HUC} \label{sec:MILP}


The 1-HUC comprises a set of units, each of which being the combination of a turbine and
a generator. 
From the hydroelectric production principle, the units discharge a water flow from
a reservoir upstream to another downstream. The reverse is also
possible for some plants equipped with pumping units. 
The horizon $T$ is discretized in time periods. 
The plant operates
at a finite number $|\setop{}|$ of so called operating points defined as pairs $(P_i, D_i)$, $i\in \{1, ..., |\setop{}|\}$, of the generated
power and the corresponding water flow. Note that operating points do not depend on the time period.
The production rate is subject
to upper and lower bounds, namely ramping constraints defined by bounds $R_u$ and $R_d$, respectively.
Min-up/min-down constraints apply, in the sense that when a plant is ramping up (resp. down), 
a minimum time delay $L$
is imposed before ramping down (resp. up).
Each reservoir has an initial volume, as well
as window constraints, defined by a minimum and
maximum volume per time period. The time-dependent maximum and minimum
capacities make it possible to account for additional intake of water in the reservoirs and also to set target volumes for specific time periods.
 A target volume is defined as the volume to be met at the end of a predefined period, typically a day, for water management policies. The case of a price-taker revenue maximization
problem is considered. 
For each time period $t$ and operating point $i$, the associated revenue is  the price $\valop{t}{i}$, usually provided by the Lagrangian decomposition algorithm run daily at EDF, and defined as follows: $ \valop{t}{i}=\valpui{t}\pui{i}+(\Phi^2-\Phi^1)D_i$ where $\valpui{t}$ is a time-dependent unitary value of the power, and $\Phi^1$ (resp. $\Phi^2$) is the unitary value of the water in the upstream (resp. downstream) reservoir at the end of the time horizon.

In the sequel a mathematical programming formulation for the considered 1-HUC problem is introduced.
Without loss of generality the proposed formulation does not feature pumps, since  an instance of the 1-HUC featuring both pumps and turbines can be cast as an instance featuring only turbines~\cite{van2021decomposition}.
Besides  complementary reformulations along with bound tightening techniques from \cite{heintzmann2023efficient} are used to further enhance the formulation.

Let $x_{t}^{i}$ be the binary variable indicating that operating point $i$ is active at time period $t$.
The formulation $\model{}$ is as follows.
\begin{align}
    \text{max } & \sum_{t=1}^{T}\sum_{i\in \setop{}} \valop{t}{i} \cdot x_{t}^{i} \label{eq:model_obj}
\end{align}
\vspace{-.2in}
\begin{align}
    \text{s.t.} \quad
    & \binfstar{t} \leq \sum_{t'=1}^{t} \sum_{i\in \setop{}}\deb{i}x_{t'}^{i} \leq \bsupstar{t} & \forall t \leq T \label{eq:model_flow}\\
    & x_{t}^{i} \geq x_{t}^{i+1} & \forall t \leq T, \forall i \leq |\setop|-1 \label{eq:model_precedence_turbine}\\
    & -\rampdown{} \leq \sum_{i\in \setop{}}\deb{i}x_{t}^{i}-\sum_{i\in \setop{}}\deb{i}x_{t-1}^{i} \leq \rampup{} & \forall t \in \{2,\ldots, T\} \label{eq:model_ramp}\\
    & \sum_{t'=t-\lenpallier{}+1}^{t} v_{t'}^{i}\leq x_{t}^{i} & \forall t \in \{\lenpallier{}, \ldots, T\}, \forall i \in \setop{} \label{eq:model_min_up}\\
    & \sum_{t'=t-\lenpallier{}+1}^{t} v_{t'}^{i} \leq 1-x_{t-\lenpallier{}}^{i} & \forall t \in \{\lenpallier{}, \ldots, T\}, \forall i \in \setop{} \label{eq:model_min_down}\\
    & v_{t}^{i} \geq x_{t}^{i}-x_{t-1}^{i} & \forall t \in \{2, \ldots, T\}, \forall i \in \setop{} \label{eq:model_additional_min_up_down} \\
    & x_{t}^{i} \in \{0,1\} & \forall t \leq \nbt, \forall i \in \setop{} \label{eq:model_binary_x}\\
    & v_{t}^{i} \in \{0,1\} & \forall \; 2\leq t \leq \nbt, \forall i \in \setop{} \label{eq:model_binary_v}
\end{align}

Objective-function \eqref{eq:model_obj} represents the total revenue.
Inequalities \eqref{eq:model_flow} correspond to the tight upper bound $\bsupstar{t}$ and lower bound $\binfstar{t}$ for the amount of water used since time period 1, for each time period $t\leq T$.
Inequalities \eqref{eq:model_precedence_turbine} are the order constraints between the operating points.
Inequalities \eqref{eq:model_ramp} express the ramping constraints.
Inequalities \eqref{eq:model_min_up} to \eqref{eq:model_additional_min_up_down} are the min-up/down constraints as formulated in \cite{rajan2005minimum,rottner2025generalized} with additional binary variables $v_{t}^{i}$.

In this formulation, constraints \eqref{eq:model_precedence_turbine} to \eqref{eq:model_binary_v} can be modeled as a longest-path problem in an acyclic graph.  
Constraints \eqref{eq:model_flow} for $t'<T$ are additional constraints \eqref{additional_constraints}, and constraint \eqref{eq:model_flow} for $t' = T$, which  corresponds to the target volume, is the main additional constraint \eqref{main_additional_constraint}.

\paragraph{}
Note that the 1-HUC features a Nested Multiple-Choice Knapsack (NMCKP) structure: indeed, for each $t$, constraint \eqref{eq:model_flow} can be interpreted as a knapsack constraint with both an upper and a lower bound. The multiple-choice aspect comes from the fact that for a given time $t'$, variables $x^i_{t'}$ are incremental, therefore exactly one operating point is chosen per time period. As for the nested structure, it arises from constraints \eqref{eq:model_flow} where the amount of water used $\sum_{t=1}^{t'} \sum_{i\in \setop{}}\deb{i}x_{t}^{i}$ increases with $t$.

\subsection{Related work on the HUC}
\label{sec:litterature_1-huc}

The second phase of BORWin can be identified to a dynamic programming method. Such an approach has been widely used to solve the HUC. For the sake of conciseness, the presentation of the related work on the HUC in this section is focused  on dynamic programming.

In \cite{renaud1993daily} a two-phase approach is presented for solving the HUC problem.
In the first phase, an LP for the HUC problem is considered, where  various constraints are omitted.
In the second phase, a dynamic programming algorithm is devised to solve the 1-HUC problem for each plant of a valley. The aim is to get a solution as close as possible to the LP solution, while taking into account omitted constraints.
The underlying graph relies on a discretization of the volume for each reservoir, thus leading to hundreds of possible values.
As there is also a discretization of the time in periods, a vertex is defined for each pair of  values for the volume and time and an arc between any two consecutive vertices in terms of time periods.
A Bellman-Ford algorithm \cite{bellman1958routing} is used to find a path in this graph.
Note that even a fine discretization of the volume discards a lot of realistic paths. 
In this paper, the graph does not rely on a discretization of the volume.


In \cite{perez2010optimal} a dynamic programming approach is defined, aiming to solve the 1-HUC with a target volume. 
The underlying graph 
is similar to the one in \cite{renaud1993daily}.
In order to match a discrete value of the volume and keep the instance feasible, the target volume is relaxed. Then the corresponding arcs are built
 by generating the feasible paths to reach the target volume, while satisfying the bounds on the volume at each time period.
A backward algorithm is then proposed, aiming to maximize the value of the generated power.
As we consider the 1-HUC with and also without target volume, this approach may not be practical.
Indeed, if no target volumes are set, then one needs to repeat the algorithm for any possible volume value at the last time period, which can be exponential.

In \cite{van2021decomposition}, a
decomposition algorithm is considered.  The main idea is to decompose the valley into single plants, thus reducing to 1-HUC in each sub-problem. The 1-HUC problem is transformed as a WCSPP to be solved with dynamic programming~\cite{barrett2008engineering}.
For the underlying graph, there is a vertex for each operating point and each time period and an arc between any two consecutive vertices in time provided the corresponding operating points satisfy the ramping constraints.
When the volumes are not restrictive, the 1-HUC is solved by a shortest path algorithm, otherwise by a labeling algorithm as mentioned in Section \ref{sec:WCSPP}.  In the latter case, 
the labeling algorithm is an adaptation of an RCSPP algorithm \cite{barrett2008engineering}, which  considers a window lower bound 
.
It is mentioned that the use of dominant rules is crucial to efficiently reduce all possible labels. Such rules are not effective between any two labels when one satisfies the lower bound while the other does not.  This would be the case when considering window constraints for the target volume. 

In~\cite{Alexandre2024}
an alternative $A^*$-based exact
algorithm, denoted HA$^*$, is proposed by defining 1-HUC as a longest-path problem in a graph with a pseudo-polynomial number of vertices. Unfortunately, the upper bound used to guide the search is seriously weakened when min-up, min-down and ramping contraints are present.
In~\cite{heintzmann2023efficient}, HA$^*$ is evaluated on 1-HUC instances without such constraints.
 Numerical experiments highlight significant improvements in the case of a window constraint on the target volume compared to state-of-the-art methods.

The 1-HUC problem is related to the subproblem dedicated to a thermal plant. It shares the ramping constraints, but obviously there is no lower bound on the volume and the definition of the min-up/min-down constraints differs slightly. 
In \cite{parmentier2018EDF}, the subproblem dedicated to each thermal plant is represented  as an RCSPP solved with dynamic programming.
Even if the underlying graph could be adapted to represent the min-up/min-down constraints for hydro plants, it would not account for window contraints on the volume, which appears to be the main issue for the 1-HUC problem.

To summarize, efficient dynamic programming-based approach proposed so-far for the 1-HUC mostly either ignored the window constraints or the ramping and min-up/min-down constraints. The labeling algorithm proposed in \cite{van2021decomposition} considers windows  and ramping constraints but no min-up/min-down constraints. In what follows, we show how the 1-HUC with windows, ramping and min-up/min-down constraints can be modeled as an AWCLPP making it addressable by BORWin. Compared to the labeling algorithm proposed in \cite{van2021decomposition}, BORWin uses the phase I bound to guide the search, as well as  pruning rules (\ref{pr:pruningrule}-- \ref{pr:ub}) and dominance rule \eqref{dm:dominancerule3}.

\subsection{Graph representation}
\label{sec:1-huc_graph}


The 1-HUC graph representation we consider extends the ones of \cite{van2021decomposition,heintzmann2023efficient} by taking into account min-up/down constraints and allows us to model the 1-HUC as an AWCLPP.


Let $\Gpoint{}=(\Vpoint{},\Apoint{})$ be the graph defined as follows.
Each vertex $u\in \Vpoint{}$ is defined as a triplet ($t$, $i$, $l$).
For this definition, $t$ is the time period, $i$ the operating point, and $l\in\{-\lenpallier+1, \ldots, \lenpallier-1\}$ the remaining time for min-up/down constraints to be satisfied.
Also, a source vertex and target vertex, respectively $s$=(0,0,0) and $p$=$(\nbt+1,0,0)$ are defined.
For any vertex $u=(t,i,l)\in \Vpoint\setminus\{s,p\}$ we define the window $[\minresource{u}=\binfstar{t};\maxresource{u}=\bsupstar{t}]$.
For vertex $s$, the window is $[\minresource{s}=0;\maxresource{s}=\infty]$ and for $p$ it is $[\minresource{p}=\binfstar{\nbt};\maxresource{p}=\bsupstar{\nbt}]$.

There is an arc in $\Apoint$ from each vertex ($\nbt$, $i$, $l$) towards $p$ of value $0$ and which uses $0$ resource.
For the following description of the arcs, consider a vertex $u$=($t$, $i$, $l$) with $t\in\{0, \ldots, \nbt-1\}$.

If $l \geq 1$ there is an arc $a\in \Apoint{}$ towards ($t+1$, $i$, $l-1$).
Similarly, if $l \leq - 1$ there is an arc $a\in \Apoint{}$ towards ($t+1$, $i$, $l+1$).
If $l=0$, there is an arc towards $a\in \Apoint{}$ ($t+1$, $i$, $0$)

If $l \geq 0$, for any $i'>i$ such that $\sum_{j=i+1}^{i'}\deb{i}\leq \rampup{}$ there is an arc $a\in \Apoint{}$ from $u$ to ($t+1$, $i'$, $\lenpallier{}-1$), or ($t+1$, $i'$,0).
Similarly, if $l \leq 0$, for any $i'<i$ such that $\sum_{j=i+1}^{i'}\deb{i}\leq \rampdown{}$ there is an arc $a\in \Apoint{}$ from $u$ to ($t+1$,$i'$,$-\lenpallier{}+1$).

The value of any arc towards a vertex ($t$, $i$, $l$), with $t\leq \nbt$ is $\sum_{j=0}^{i} \valop{t}{j}$, and the resource used is $\sum_{j=1}^{i}\deb{j}$.



\begin{pte}[Polynomial number of vertices of $\Gpoint{}$] \label{property:polynomialnberofverticesofGc}
    The number of vertices of $\Gpoint{}$ is polynomial with respect to the size of the instance.
\end{pte}
\begin{proof}
    See Appendix \ref{anx:HUC:graph:proofofpropertypolynomialsizegraph}.
\end{proof}

\begin{pte}
    \label{pte:graphs:feasible_path}
    Any path in graph $\Gpoint{}$ satisfies order constraints, ramping constraints and min-up/down constraints.
\end{pte}
\begin{proof}
    See Appendix \ref{anx:HUC:graph:feasiblepath}.
\end{proof}

Paths in graph $\Gpoint{}$ satisfy order constraints, ramping constraints and min-up/down constraints (see Property \ref{pte:graphs:feasible_path}) but may not satisfy the window constraints of the 1-HUC problem. Therefore, solving the AWCLPP on $\Gpoint{}$ provides an optimal solution to the 1-HUC problem. 

Note that  the NMCKP structure of the 1-HUC involves identical weights due to operating points being the same over the time periods. It follows that the weakened dominance rule \ref{dm:dominancerule3} inside the BORWin algorithm should eliminate many symmetric paths
as mentioned in Remark \ref{rmk:efficient} of Section \ref{sec:AWXLPP-application}.


The construction of graph $\Gpoint{}$ is illustrated in \textbf{Figure \ref{fig:graph_1-huc}} from \textbf{Example \ref{ex:graphs:example_graph_point}}.


\begin{example}
    \label{ex:graphs:example_graph_point}
Consider an instance of the 1-HUC problem for $\nbt{}=5$ time periods.
Accounting for the idle operating point, the unit operates on 3 operating points: ($\deb{0}=0$, $\pui{0}=0$), ($\deb{1}=6$, $\pui{1}=8$), ($\deb{2}=5$, $\pui{2}=6$).
The ramps rates are $\rampup{}=\rampdown{}=6$ and the duration of the min-up/down constraints is $\lenpallier=3$.
The bounds are $\binfstar{}=$(0, 0,  7, 18) and $\bsupstar{}=$(11, 18, 18, 18). 
The value of operating points are: $\valop{\cdot}{1}=(0, 2.8, 1)$, $\valop{\cdot}{2}=(0, -6.8, -6.2)$, $\valop{\cdot}{3}=(0, 0.4, -0.8)$, $\valop{\cdot}{4}=(0, -11.6, -9.8)$ and $\valop{\cdot}{5}=(0, 2.0, 0.4)$.

\textbf{Figure \ref{fig:graph_1-huc}} represents the graph $\Gpoint{}$ associated with this instance.
The vertices appear in black if they can be reached from $s$, and in gray otherwise.
Dotted lines separate the vertices with respect to each operating point.
For readability purposes, the values and the resource of the arcs are not represented in this graph.
As the value only depends on the time period $t$ and the operating point $i$ they are the same for any arc heading to any vertex ($t$, $i$, $l$) no matter the value of $l$.
\textbf{Table \ref{tab:graphs:example_values}} gives the value of each combination of $t$ and $i$.
Similarly, the resource only depends on the operating point $i$ which is the same for any arc heading to any vertex ($t$, $i$, $l$) no matter the value of $t$ or $l$.
For operating point 0, the resource is 0; for operating point 1, the resource is 6 and for operating points 2, the resource is 11.
\end{example}

\begin{figure}[!htb]
\centering
\includegraphics[width=\textwidth]{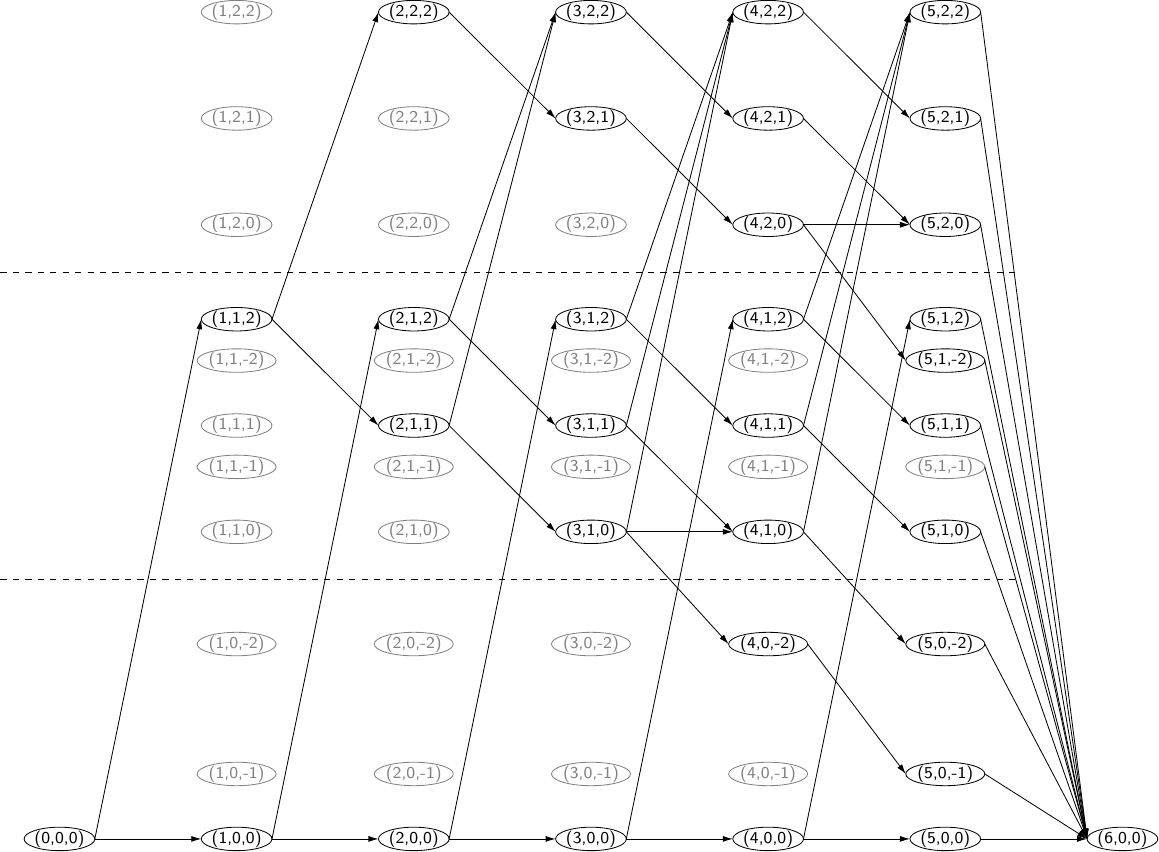}
\caption{Graph $\Gpoint$ of \textbf{Example \ref{ex:graphs:example_graph_point}}}
\label{fig:graph_1-huc}
\end{figure}

\begin{table}[htb!]
    \centering
    \begin{tabular}{c|r|r|r}
        & $i=0$ & $i=1$ & $i=2$ \\
         $t=1$ & 0 & 2.8 & 3.8 \\
         $t=2$ & 0 & -6.8 & -13.0 \\
         $t=3$ & 0 & 0.4 & -0.4 \\
         $t=4$ & 0 & -11.6 & -21.4 \\
         $t=5$ & 0 & 2.0 & 2.4 
    \end{tabular}
    \caption{Values of the arcs of graph $\Gpoint{}$ depicted in \textbf{Figure \ref{fig:graph_1-huc}}}
    \label{tab:graphs:example_values}
\end{table}

\begin{rmk}
    \label{rmk:graphs:source_vertex}
    The 1-HUC problem is solved in practice on a daily basis, and the plant is not stopped between two consecutive days.
    This means that one needs to consider the last decision of the previous day when solving the 1-HUC problem.
    This can be taken into account by modifying source vertex $s$ as follows.
    Consider the path for the previous day going through vertex $(T,i,l)$.
    One can initialize $s=(0,i,l)$.
    As such, the plant can be operating continuously from one day to another, without violating any constraint.
    Note however that modifying vertex $s$ in such a manner does not modify the structure of the graph.
    Hence, in the following we consider $s$=(0,0,0) without loss of generality.
\end{rmk}

\section{Experimental results}
\label{sec:results}

All presented results are computed on a single thread of an Intel Core i7-9850H CPU @ 2.60GHz processor of 12 cores, with Linux as operating system.
All algorithms are developed with C++ and version 12.8 of CPLEX is used.
The code for the BORWin algorithm is available online \footnote{\url{https://github.com/Eegann/BORWin}}.

A generic Hydro Instance Generator (HIG) has been provided by Dimitri Thomopulos and is available online \footnote{\url{https://www.lix.polytechnique.fr/Labo/Dimitri.Thomopulos/libraries/HIG.html}}. However, preliminary results in  \cite{heintzmann2022DP} indicate that 1-HUC instances created with HIG can be relatively straightforward for CPLEX to solve: when applied to our mathematical model, CPLEX solves most of these instances in under a second.
Preliminary results have also shown that when the price of the energy $\valpui{t}$ is very similar from one time period to another, instances of the 1-HUC are harder to solve.
These preliminary results are consistent with the observations on instances of the Knapsack Problem, which are more difficult to solve when their values are highly correlated to their weight \cite{pisinger2005hard}.

For this study, we consider instead two sets of instances derived from real-world EDF's plants and reservoirs, as they are less easy to solve with CPLEX. 
In the first set of 83 instances, denoted as set A, the initial volume and some of the minimum and maximum volume per time period have been adjusted to ensure that at least one of the bounds on the window is active for at least one time period, as in \cite{heintzmann2022DP}. 
The second set of 83 instances, denoted as set B, is built as follows.
For each instance of the first set, an instance of the second set is created with $\valpui{t}$ being a random value in $[0.95\cdot \valpui{1};1.05\cdot \valpui{1}]$.

For our computational analysis we compare four approaches: model $\model{}$ solved by CPLEX, RCSPP algorithm as described in \cite{van2018shortest}, HA* algorithm \cite{heintzmann2023efficient} and BORWin.

\textbf{Figure \ref{fig:cactus_1}} (resp. Figure \textbf{\ref{fig:cactus_2}}) shows, for each approach, the number of instances solved with respect to the time for instance set A (resp. instance set B).
\textbf{Table \ref{tab:res_borwin_a1}} to \textbf{Table \ref{tab:res_borwin_b2}} show the value of the solutions and the time required by each approach for each instance considered.
For CPLEX, we also provide the number of nodes developed, denoted as \#node, as well as the optimality gap, being opt (resp inf, sub) if the solution returned is optimal (resp. infeasible, suboptimal).
We also added the number of iterations of the second phase for BORWin, denoted as \#iter.

Clearly, these figures show that HA* and the RCSPP algorithms are not suitable to solve the 1-HUC problem.
On the contrary, BORWin appears to be very well suited as it is the most efficient approach as soon as the computational times exceed 3 seconds.
Moreover, in each set, BORWin solved 15 more instances than CPLEX  which is the current approach at EDF.
These results also confirm the preliminary results, as it appears that instances of set B are harder to solve than the ones of set A.

HA* algorithm is not suitable due to its upper bound, which does not take into account ramping or min-up/down constraints, hence being too loose to efficiently guide the enumeration.
The RCSPP algorithm is inefficient as the dominance rule is seriously weakened due to the lower bound, as stated in \textbf{Section \ref{sec:litterature_1-huc}}.
When it comes to solving the MILP with CPLEX, there are 12 instances of the first set, and 5 of the second set, where the solution returned is infeasible due to numerical errors, as it violates one of the window constraints.
This violation is in most cases below 0.1\%, but can reach up to 10\%.
For instance 64 in \textbf{Table \ref{tab:res_borwin_a2}}, the volume of the upstream reservoir violates all upper bounds from time period 29 to 96.
At time period $\nbt$, the volume reaches 29175.04 whereas the upper bound $\volmax{T}{1}$ is 26380.
Such numerical errors have been observed, and a study has shown that these errors are due to floatting-point precision of the solver \cite{bendotti2016impact,sahraoui2019real}.
In addition to the infeasible solutions, there are 2 instances of the first set, and 6 of the second set where the instances are suboptimal.
Instance 48 from set B shows a case where the value of solution obtained when solving the MILP with CPLEX is 0.009\% below the value of the optimal solution obtained with BORWin.
This is consistent with the optimality gap of CPLEX, which is $10^{-4}$ by default.
Note that it is also possible that solving the MILP with CPLEX yields a solution that is suboptimal and also infeasible.
This is the case for instance 57 from set A.

\begin{figure}[htb!]
    \centering
    \includegraphics[scale=1.5]{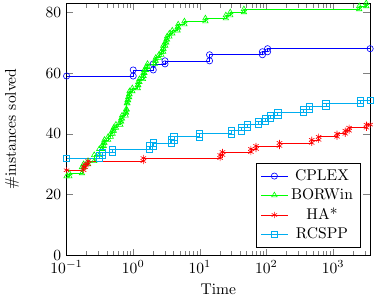}
    \caption{Number of instances solved by each approach with respect to the time for instance set A}
    \label{fig:cactus_1}
\end{figure}
    
\begin{figure}[htb!]
    \centering
    \includegraphics[scale=1.5]{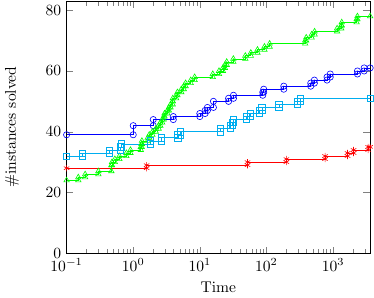}
    \caption{Number of instances solved by each approach with respect to the time for instance set B}
    \label{fig:cactus_2}
\end{figure}

\section{Conclusion}
\label{sec:conclusion}

In this paper, we proposed the BORWin algorithm to solve exactly any mixed integer maximization problem featuring special structure  $(P)$ characterized by the presence of several additional constraints. 
The first phase provides a family of lower bounds for $(P)$ based on a bi-objective relaxation of the additional constraints. It is shown that the obtained bounds strongly relate to Lagrangian dual bounds.
In the case where the special structure within $(P)$ is an acyclic longest-path, problem $(P)$ can be identified to a window constrained acyclic longest-path problem. For such problems, the second phase exploits these upper bounds to perform an efficient label extension procedure.
Moreover, if $(P)$ has other underlying structures for which efficient algorithms exist, it is shown that complementary upper bounds can be used to still accelerate the second phase. A typical example is when the additional constraints have a knapsack structure. 
Experimental results are provided for the 1-HUC, which features both longest-path and nested multiple-choice knapsack structures. 
On the easiest instances solved by Cplex in less than a second, BORWin remains a bit slower as it takes up to 3 seconds. However, on more difficult instances, it shows that BORWin algorithm outperforms other approaches in the sense that it is able to solve significantly more instances within time limits.

As a perspective, BORWin algorithm could be used for other applications showing a similar structure, for example smart charging of one electric vehicule under battery capacity constraints.
Another perspective would be to speed-up BORWin on the easiest instances by hybridization with other approaches, for example with an efficient MILP solver.
The resulting algorithm could then be used as a pricing algorithm within a decomposition framework to solve large-scale instances of problems involving many blocks with structure $(P)$. For example, this is the case for HUC instances with many reservoirs and power-plants interconnected in a water valley, or smart charging of a large fleet of electric vehicles coupled by grid constraints.

Finally, BORWin algorithm opens new optimization perspectives at EDF, as it could be integrated in operational tools to replace MILP solvers for single hydro units optimization. Moreover, window constraints appear in many other industrial contexts at EDF (optimization of gas units, optimization of electricity storage, ...) and BORWin algorithm could prove very useful to help speed-up classical RCSPP algorithms currently used.



\printbibliography

\appendix

\section{Proofs}

\subsection{Proof of \textbf{Property \ref{pte:graphs_twophases:valid_upper_bound1}}}
\label{anx:graphs_twophases:valid_upper_bound1}

\begin{proof}
   Consider a solution $x_3$ of $(BOR)$ and suppose that  $\agg{\delta}{x_3}> \agg{\delta}{x_{1}}$. 

    Consider the case $\arcvalobj{1}{x_{1}}<\arcvalobj{1}{x_{3}}<\arcvalobj{1}{x_{2}}$ illustrated by \textbf{Figure \ref{fig:graphs_twophases:upper_bound_aggregated_1}}.
    In such a case, $x_{1}$ and $x_{2}$ cannot be two consecutive solutions in $S$ by definition.
    Indeed, either $x_{3}$, or a Pareto-supported solution dominating $x_{3}$ is between $x_{1}$ and $x_{2}$ in $S$.

    Consider the case $\arcvalobj{1}{x_{1}}< \arcvalobj{1}{x_{2}}< \arcvalobj{1}{x_{3}}$ illustrated by \textbf{Figure \ref{fig:graphs_twophases:upper_bound_aggregated_2}}. Note in this case that $\arcvalobj{2}{x_{1}}\geq \arcvalobj{2}{x_{2}}\geq \arcvalobj{2}{x_{3}}$ otherwise solutions $x_1$ and $x_2$ would not be Pareto-supported. 
    Let $\delta'$ be such that $\agg{\delta'}{x_{1}}=\agg{\delta'}{x_{3}}$.
    We suppose that $\agg{\delta'}{x_{2}}\geq\agg{\delta'}{x_{1}}$ to show a contradiction.
    We then know the following:
    \begin{align*}
       \arcvalobj{1}{x_{1}}+\delta\cdot \arcvalobj{2}{x_{1}}< \arcvalobj{1}{x_{3}}+\delta\cdot \arcvalobj{2}{x_{3}}\\
         \arcvalobj{1}{x_{1}}+\delta'\cdot \arcvalobj{2}{x_{1}} \leq \arcvalobj{1}{x_{2}}+\delta'\cdot \arcvalobj{2}{x_{2}}
    \end{align*}
    We can deduce $\delta \cdot (\arcvalobj{2}{x_{1}}-\arcvalobj{2}{x_{3}})< (\arcvalobj{1}{x_{3}}-\arcvalobj{1}{x_{1}})$ and $\delta'\cdot (\arcvalobj{2}{x_{1}}-\arcvalobj{2}{x_{2}})\leq  (\arcvalobj{1}{x_{2}}-\arcvalobj{1}{x_{1}})$.
    Using equalities:
        \begin{align*}
         \arcvalobj{1}{x_{1}}+\delta\cdot \arcvalobj{2}{x_{1}} =  \arcvalobj{1}{x_{2}}+\delta\cdot \arcvalobj{2}{x_{2}}\\
        \arcvalobj{1}{x_{1}}+\delta'\cdot \arcvalobj{2}{x_{1}} =  \arcvalobj{1}{x_{3}}+\delta'\cdot \arcvalobj{2}{x_{3}}\\
       \end{align*} we obtain from the inequality featuring $\delta$:
    \begin{align*}
        &\frac{\arcvalobj{2}{x_{1}}-\arcvalobj{2}{x_{3}}}{\arcvalobj{2}{x_{1}}-\arcvalobj{2}{x_{2}}}<\frac{\arcvalobj{1}{x_{3}}-\arcvalobj{1}{x_{1}}}{\arcvalobj{1}{x_{2}}-\arcvalobj{1}{x_{1}}}
    \end{align*}
    and from the inequality featuring $\delta'$:
    \begin{align*}
        & \frac{\arcvalobj{2}{x_{1}}-\arcvalobj{2}{x_{3}}}{\arcvalobj{2}{x_{1}}-\arcvalobj{2}{x_{2}}} \geq \frac{\arcvalobj{1}{x_{3}}-\arcvalobj{1}{x_{1}}}{\arcvalobj{1}{x_{2}}-\arcvalobj{1}{x_{1}}}
    \end{align*}
    which is a contradiction.
    Hence $\arcvalobj{1}{x_{1}}< \arcvalobj{1}{x_{2}}< \arcvalobj{1}{x_{3}}$ implies that $\agg{\delta'}{x_{2}} < \agg{\delta'}{x_{1}}=\agg{\delta'}{x_{3}}$.
    \textbf{Lemma \ref{lem:graphs_twophases:not_pareto}} shows that in such a case $x_{2}$ cannot be Pareto-supported, which is a contradiction with $x_2 \in S$.

    A similar contradiction can be obtained in the case $\arcvalobj{1}{x_{3}}< \arcvalobj{1}{x_{1}} < \arcvalobj{1}{x_{2}}$.
\end{proof}

\begin{figure}[htb!]
    \centering
    \begin{subfigure}{0.49\textwidth}
    \centering
    \includegraphics[width=\linewidth]{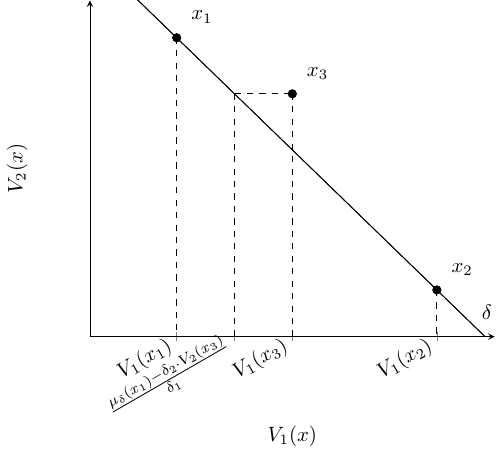}
    \caption{Case  $\arcvalobj{1}{x_{1}}<\arcvalobj{1}{x_{3}}<\arcvalobj{1}{x_{2}}$}
    \label{fig:graphs_twophases:upper_bound_aggregated_1}
    \end{subfigure}
    \begin{subfigure}{0.49\textwidth}
    \centering
    \includegraphics[width=\linewidth]{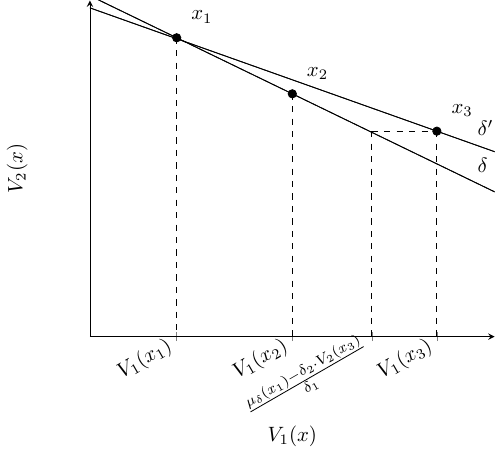}
    \caption{Case  $\arcvalobj{1}{x_{1}}<\arcvalobj{1}{x_{2}}<\arcvalobj{1}{x_{3}}$}
    \label{fig:graphs_twophases:upper_bound_aggregated_2}
    \end{subfigure}
    \caption{Illustrations for \textbf{Property \ref{pte:graphs_twophases:valid_upper_bound1}}}
\end{figure}

\subsection{Proof of \textbf{Theorem \ref{thm:graphs_twophases:best_upper_bound}}}
\label{anx:graphs_twophases:best_upper_bound}
\begin{proof}
    Consider another pair $(x_{1}',x_{2}')$ of consecutive solutions in $S$, and $\delta'$ such that $\agg{\delta'}{x_{1}'}=\agg{\delta'}{x_{2}'}$.
Illustations can be found in 
\textbf{Figure \ref{fig:graphs_twophases:best_pair_1}} and \textbf{Figure \ref{fig:graphs_twophases:best_pair_2}}.
   
    Suppose $\bsupvalobj{1}{x_{1}'}{\delta'} <\bsupvalobj{1}{x_{1}}{\delta}$, $i.e.$, $${\agg{\delta'}{x_{1}'}}-{\agg{\delta}{x_{1}}}< ({\delta'}-{\delta} ) \cdot \beta$$ 
      
      Using \textbf{Property} \ref{pte:graphs_twophases:valid_upper_bound1}, we have for any solution $x$ of $(BOR)$:
      $${\agg{\delta'}{x}} = \arcvalobj{1}{x} + \delta' \arcvalobj{2}{x}  \leq {\agg{\delta'}{x'_{1}}}$$ 
      
    Replacing $\arcvalobj{1}{x}$ by $\agg{\delta}{x} - \delta \arcvalobj{2}{x}$, we obtain 
     $$
        \left(  {\delta'}-{\delta}\right) \cdot \arcvalobj{2}{x} \leq {\agg{\delta'}{x_{1}'}}-{\agg{\delta}{x}}$$
      

Therefore, applying to $x=x_1$ and $x=x_2$, and using ${\agg{\delta}{x_{1}}}={\agg{\delta}{x_{2}}}$:
\begin{align*}
\left(  {\delta'}-{\delta}\right) \cdot \arcvalobj{2}{x_{1}} \leq {\agg{\delta'}{x_{1}'}}-{\agg{\delta}{x_{1}}} < ({\delta'}-{\delta} ) \cdot \beta \\
\left(  {\delta'}-{\delta}\right) \cdot \arcvalobj{2}{x_{2}} \leq {\agg{\delta'}{x_{1}'}}-{\agg{\delta}{x_{1}}} < ({\delta'}-{\delta} ) \cdot \beta
\end{align*}

As $\arcvalobj{2}{x_{2}} <\beta \leq \arcvalobj{2}{x_{1}}$, then there is a contradiction.

\end{proof}

\begin{figure}[htb!]
    \centering
    \begin{subfigure}{0.48\textwidth}
    \centering
    \includegraphics[width=\linewidth]{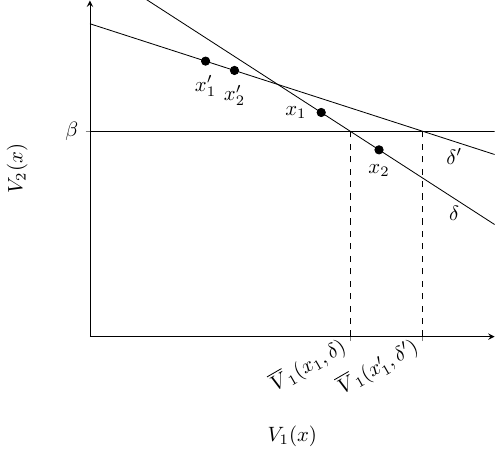}
    \caption{Case $\arcvalobj{1}{x_{1}}<\arcvalobj{1}{x_{2}}\leq \arcvalobj{1}{x_{1}'}<\arcvalobj{1}{x_{2}'}$}
    \label{fig:graphs_twophases:best_pair_1}
    \end{subfigure}
    \begin{subfigure}{0.48\textwidth}
    \centering
    \includegraphics[width=\linewidth]{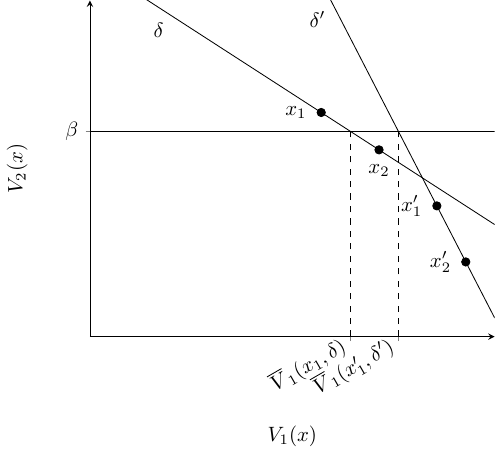}
    \caption{Case $\arcvalobj{1}{x_{1}'}<\arcvalobj{1}{x_{2}'}\leq \arcvalobj{1}{x_{1}}<\arcvalobj{1}{x_{2}}$}
    \label{fig:graphs_twophases:best_pair_2}
    \end{subfigure}
    \caption{Illustrations for \textbf{Theorem \ref{thm:graphs_twophases:best_upper_bound}}}
\end{figure}

\subsection{Proof of \textbf{Property \ref{property:polynomialnberofverticesofGc}}}
\label{anx:HUC:graph:proofofpropertypolynomialsizegraph}

\begin{proof}
    There are $ (\nbop{}+1) \cdot (2 \lenpallier{} - 1)$ vertices per time period.
    Hence, there are $\nbt{} \cdot (\nbop+1) \cdot (2 \lenpallier{} - 1) + 2$ vertices in $\Gpoint{}$ in total.
    As $\lenpallier{} \leq \nbt{}$, we deduce that the number of vertices is at most $\nbt{} \cdot (\nbop{}+1) \cdot (2 \nbt{} - 1) + 2$
    
    From the definition of the HUC problem, there are data for each operating points, such as $\deb{i}$, but also for each time period, such as $\bsup{t}$.
    Hence the size of the instance is of $\mathcal{O}(\nbt{} + \nbop{})$.
    
    As the number of verticies is of $\mathcal{O}(\nbt{}^{2}\cdot \nbop{})$ the number of vertices of $\Gpoint{}$ is polynomial with respect to the size of the instance.
\end{proof}

\subsection{Proof of \textbf{Property \ref{pte:graphs:feasible_path}}}
\label{anx:HUC:graph:feasiblepath}

\begin{proof}
    Consider an arc between ($t$, $i$, $l$) and ($t+1$, $i'$, $l'$).
    Clearly, in the case $i=i'$, the ramping constraints are satisfied if such an arc exists in $\Gpoint{}$.
    In the case $i< i'$, such an arc only exists if $\sum_{j=i+1}^{i}\deb{j} \leq \rampdown{}$ by construction of $\Gpoint{}$.
    In the case $i>i'$, such an arc only exists if $\sum_{j=i'}^{i-1}\deb{j} \leq \rampup{}$ by construction of $\Gpoint{}$.
    Consequently, ramping constraints are always satisfied.

    By construction of $\Gpoint{}$, from a vertex ($t$, $i$, $l$), the only way to reach operating point $i'>i$ is through the arc towards ($t+1$, $i'$, $\lenpallier{}-1$).
    Also by construction of $\Gpoint{}$, from ($t+1$, $i'$, $\lenpallier{}-1$,u), one needs to go through
    all vertices ($t+1+\tau$, $i'$, $\lenpallier{}-1-\tau$), with $\tau \in \{1,\ldots, \lenpallier{}-2\}$, and ($t+\lenpallier{}$, $i'$, $0$) in order to reach operating point $i''<i'$.
    In total, a path must go through at least $1+\lenpallier{}-2+ 1=\lenpallier{}$ vertices with operating point $i'$ before heading to a vertex with operating point $i''<i'$.
    Hence, the min-up constraint is satisfied.
    We can prove in a similar way that the min-down constraints are satisfied.
\end{proof}

\section{Numerical results}

In this appendix, detailed experimental results are presented (see Section \ref{sec:results}).
Note that in Tables \ref{tab:res_borwin_a1} and \ref{tab:res_borwin_b1}, instance 21 shows an optimal solution value equal to zero. Indeed, in this instance the initial reservoir volume is equal to the minimum volume of each time period and there is no intake. Therefore, all the water flows must trivially be zero in any solution.

\begin{table}[!htb]
    \centering
    \begin{adjustbox}{max width=1\textwidth, center=\textwidth}
    \begin{tabular}{r||S[table-format=7.2]|S[table-format=1.2]|S[table-format=4.2]|r||S[table-format=7.2]|S[table-format=4.2]||S[table-format=7.2]|S[table-format=4.2]||S[table-format=7.2]|r|S[table-format=4.2]}
        & \multicolumn{4}{c}{CPLEX} & \multicolumn{2}{c||}{RCSPP}& \multicolumn{2}{c||}{HA*} & \multicolumn{3}{c}{BORWin}\\
        \text{instance} & \text{value} & \text{gap} & \text{time} & \text{\#nodes} & \text{value} & \text{time} & \text{value} & \text{time} & \text{value} & \text{\#iter} & \text{time} \\
            1 & -27092.4 & \text{opt} & 14.0 & 24213 & \text{-} & 3969.29 & \text{-} & 3600.33 & -27092.4 & 2950 & 1.43\\
            2 & 42334.7 & \text{opt} & 1.0 & 9653 & 42334.7 & 3.77 & \text{-} & 3600.41 & 42334.7 & 2482 & 1.4\\
            3 & 2480.68 & \text{opt} & 104.0 & 97567 & \text{-} & 3842.11 & \text{-} & 3600.36 & 2480.68 & 5475 & 3.49\\
            4 & 1556.68 & \text{opt} & 0.0 & 531 & \text{-} & 3806.62 & \text{-} & 3600.48 & 1556.68 & 316 & 0.51\\
            5 & 108120.0 & \text{opt} & 0.0 & 863 & \text{-} & 3620.11 & \text{-} & 3600.34 & 108120.0 & 860 & 1.64\\
            6 & -2550.61 & \text{opt} & 0.0 & 630 & \text{-} & 3963.0 & \text{-} & 3600.44 & -2550.61 & 310 & 0.87\\
            7 & 3031.07 & \text{inf} & 0.0 & 77 & \text{-} & 3857.97 & \text{-} & 3600.3 & 3031.07 & 2753 & 2.64\\
            8 & 13903.4 & \text{inf} & 0.0 & 2487 & \text{-} & 3654.9 & \text{-} & 3600.28 & 13903.4 & 9644 & 12.12\\
            9 & -81458.9 & \text{inf} & 22.0 & 9845 & \text{-} & 3791.48 & \text{-} & 3600.2 & -81458.9 & 4507 & 5.87\\
            10 & -622046.0 & 0.17 & 3599.0 & 164270 & \text{-} & 3618.49 & \text{-} & 3600.2 & -622046.0 & 784 & 3.22\\
            11 & 41736.5 & \text{opt} & 0.0 & 0 & 41736.5 & 112.82 & \text{-} & 3600.52 & 41736.5 & 17 & 1.56\\
            12 & -32130.1 & \text{opt} & 87.0 & 10279 & \text{-} & 3798.79 & \text{-} & 3600.14 & -32130.1 & 582 & 3.86\\
            13 & -185784.0 & 1.58 & 3599.0 & 184953 & \text{-} & 3699.51 & \text{-} & 3600.19 & -185784.0 & 264 & 2.43\\
            14 & 19048.7 & \text{opt} & 3.0 & 2537 & 19048.7 & 29.5 & \text{-} & 3600.39 & 19048.7 & 2421 & 0.77\\
            15 & 4876.22 & \text{opt} & 0.0 & 0 & 4876.22 & 0.01 & 4876.22 & 0.07 & 4876.22 & 7 & 0.08\\
            16 & 12856.2 & \text{opt} & 0.0 & 1155 & 12856.2 & 0.06 & 12856.2 & 1738.4 & 12856.2 & 1832 & 1.17\\
            17 & 14970.0 & \text{opt} & 0.0 & 435 & 14970.0 & 0.34 & 14970.0 & 606.87 & 14970.0 & 1143 & 0.37\\
            18 & 12647.3 & \text{opt} & 0.0 & 1801 & 12647.3 & 0.07 & \text{-} & 3600.33 & 12647.3 & 1845 & 0.66\\
            19 & 9276.1 & \text{opt} & 0.0 & 159 & 9276.1 & 0.04 & 9276.1 & 157.2 & 9276.1 & 815 & 0.31\\
            20 & 12600.0 & \text{opt} & 0.0 & 45 & 12600.0 & 0.31 & 12600.0 & 69.8 & 12600.0 & 601 & 0.34\\
            21 & 0.0 & \text{opt} & 0.0 & 0 & 0.0 & 0.0 & 0.0 & 0.01 & 0.0 & 1 & 0.01\\
            22 & 21996.5 & \text{opt} & 0.0 & 0 & 21996.5 & 0.04 & 21996.5 & 0.01 & 21996.5 & 0 & 0.0\\
            23 & 96870.5 & \text{opt} & 14.0 & 32704 & 96870.5 & 9.82 & 96870.5 & 3167.85 & 96870.5 & 2897 & 3.02\\
            24 & 18974.6 & \text{opt} & 0.0 & 0 & 18974.6 & 0.02 & 18974.6 & 0.01 & 18974.6 & 0 & 0.0\\
            25 & 197441.0 & \text{opt} & 0.0 & 0 & 197441.0 & 2.0 & 197441.0 & 477.82 & 197441.0 & 2058 & 2.12\\
            26 & 21588.0 & \text{opt} & 0.0 & 0 & 21588.0 & 0.04 & 21588.0 & 0.01 & 21588.0 & 5 & 0.36\\
            27 & 95900.6 & \text{opt} & 0.0 & 0 & 95900.6 & 0.02 & 95900.6 & 0.01 & 95900.6 & 0 & 0.01\\
            28 & 102450.0 & \text{opt} & 0.0 & 0 & 102450.0 & 0.02 & 102450.0 & 0.01 & 102450.0 & 0 & 0.01\\
            29 & 72778.7 & \text{inf} & 4.0 & 11675 & \text{-} & 3690.14 & \text{-} & 3600.52 & 72778.7 & 31422 & 28.62\\
            30 & -5397.97 & \text{opt} & 0.0 & 0 & -5397.97 & 0.01 & -5397.97 & 0.01 & -5397.97 & 0 & 0.0\\
            31 & 705.98 & \text{opt} & 0.0 & 0 & 705.98 & 0.01 & 705.98 & 0.01 & 705.98 & 6 & 0.06\\
            32 & 24792.1 & \text{opt} & 0.0 & 0 & 24792.1 & 0.01 & 24792.1 & 0.01 & 24792.1 & 0 & 0.0\\
            33 & 3079.3 & \text{opt} & 0.0 & 0 & 3079.3 & 0.05 & 3079.3 & 0.01 & 3079.3 & 0 & 0.0\\
            34 & 62604.5 & \text{opt} & 0.0 & 0 & 62604.5 & 355.76 & 62604.5 & 1147.09 & 62604.5 & 344 & 0.64\\
            35 & 100832.0 & \text{opt} & 0.0 & 0 & 100832.0 & 0.02 & 100832.0 & 0.01 & 100832.0 & 0 & 0.01\\
            36 & 28104.4 & \text{opt} & 0.0 & 0 & 28104.4 & 0.49 & 28104.4 & 0.21 & 28104.4 & 17 & 0.21\\
            37 & 421369.0 & \text{opt} & 0.0 & 177 & \text{-} & 3723.65 & \text{-} & 3600.41 & 421369.0 & 95 & 0.46\\
            38 & -4281.16 & \text{opt} & 0.0 & 0 & -4281.16 & 95.53 & \text{-} & 3600.49 & -4281.16 & 1776 & 1.47\\
            39 & -31910.4 & \text{opt} & 0.0 & 4955 & \text{-} & 3658.38 & \text{-} & 3600.35 & -31910.4 & 3130 & 3.14\\
            40 & 18323.5 & \text{opt} & 0.0 & 71 & 18323.5 & 42.47 & \text{-} & 3600.58 & 18323.5 & 594 & 0.84\\
            41 & -2480.26 & \text{opt} & 0.0 & 0 & -2480.26 & 0.0 & -2480.26 & 0.01 & -2480.26 & 26 & 0.11\\
            42 & 4529.62 & \text{opt} & 0.0 & 0 & 4529.62 & 0.03 & 4529.62 & 0.01 & 4529.62 & 0 & 0.0\\
    \end{tabular}
    \end{adjustbox}
    \caption{Performance of model $\model{}$ solved with CPLEX, the RCSPP algorithm, HA* and BORWin on instances 1 to 42 of set A}
    \label{tab:res_borwin_a1}
\end{table}

\begin{table}[!htb]
    \centering
    \begin{adjustbox}{max width=1\textwidth, center=\textwidth}
    \begin{tabular}{r||S[table-format=7.2]|S[table-format=1.2]|S[table-format=4.2]|r||S[table-format=7.2]|S[table-format=4.2]||S[table-format=7.2]|S[table-format=4.2]||S[table-format=7.2]|r|S[table-format=4.2]}
        & \multicolumn{4}{c}{CPLEX} & \multicolumn{2}{c||}{RCSPP}& \multicolumn{2}{c||}{HA*} & \multicolumn{3}{c}{BORWin}\\
        \text{instance} & \text{value} & \text{gap} & \text{time} & \text{\#nodes} & \text{value} & \text{time} & \text{value} & \text{time} & \text{value} & \text{\#iter} & \text{time} \\
            43 & 269050.0 & \text{opt} & 0.0 & 0 & 269050.0 & 0.03 & 269050.0 & 0.01 & 269050.0 & 0 & 0.0\\
            44 & 919601.0 & \text{sub} & 5.0 & 23903 & \text{-} & 4535.04 & \text{-} & 3600.5 & 919605.0 & 2198 & 1.19\\
            45 & 287722.0 & \text{inf} & 0.0 & 26 & \text{-} & 3689.87 & \text{-} & 3600.36 & 287722.0 & 72 & 0.68\\
            46 & 52834.5 & \text{opt} & 0.0 & 0 & 52834.5 & 0.07 & 52834.5 & 0.04 & 52834.5 & 0 & 0.01\\
            47 & 459095.0 & \text{opt} & 0.0 & 0 & 459095.0 & 0.04 & 459095.0 & 0.01 & 459095.0 & 0 & 0.0\\
            48 & 462869.0 & \text{inf} & 1.0 & 1571 & \text{-} & 3780.74 & \text{-} & 3600.45 & 462869.0 & 693 & 0.82\\
            49 & 42341.8 & \text{opt} & 0.0 & 0 & 42341.8 & 0.06 & 42341.8 & 0.03 & 42341.8 & 0 & 0.0\\
            50 & 4905.83 & \text{opt} & 1.0 & 8813 & \text{-} & 3608.06 & \text{-} & 3600.43 & 4905.83 & 2614 & 2.18\\
            51 & 40192.6 & \text{opt} & 0.0 & 0 & 40192.6 & 0.01 & 40192.6 & 0.02 & 40192.6 & 0 & 0.0\\
            52 & 48378.6 & \text{opt} & 0.0 & 280 & \text{-} & 3678.46 & \text{-} & 3600.39 & 48378.6 & 784 & 0.55\\
            53 & 28142.9 & \text{opt} & 0.0 & 0 & 28142.9 & 1.75 & 28142.9 & 0.19 & 28142.9 & 9 & 0.03\\
            54 & 3838650.0 & \text{sub} & 0.0 & 0 & 3838990.0 & 75.46 & \text{-} & 3600.46 & 3838990.0 & 1687 & 0.78\\
            55 & 9577.87 & \text{opt} & 0.0 & 425 & \text{-} & 3606.21 & \text{-} & 3600.73 & 9577.87 & 900 & 0.8\\
            56 & 28471.9 & \text{opt} & 0.0 & 0 & 28471.9 & 0.01 & 28471.9 & 0.03 & 28471.9 & 0 & 0.01\\
            57 & 3931900.0 & \text{inf} & 0.0 & 1177 & \text{-} & 3716.47 & \text{-} & 3600.29 & 3931920.0 & 287191 & 2468.04\\
            58 & 750878.0 & \text{opt} & 0.0 & 544 & \text{-} & 3637.65 & \text{-} & 3600.58 & 750878.0 & 15667 & 24.59\\
            59 & 19159.0 & \text{opt} & 0.0 & 7 & 19159.0 & 51.78 & 19159.0 & 20.16 & 19159.0 & 31 & 0.26\\
            60 & 25720.8 & \text{opt} & 0.0 & 0 & 25720.8 & 0.05 & 25720.8 & 0.01 & 25720.8 & 0 & 0.0\\
            61 & -19718.7 & \text{opt} & 0.0 & 713 & \text{-} & 3686.43 & \text{-} & 3600.6 & -19718.7 & 1555 & 0.88\\
            62 & 12003.3 & \text{opt} & 0.0 & 30 & 12003.3 & 438.36 & 12003.3 & 21.86 & 12003.3 & 109 & 0.49\\
            63 & 935608.0 & \text{opt} & 0.0 & 0 & 935608.0 & 0.06 & 935608.0 & 0.02 & 935608.0 & 0 & 0.0\\
            64 & 124600.0 & \text{opt} & 2.0 & 8219 & \text{-} & 3624.53 & \text{-} & 3600.48 & 124600.0 & 15577 & 45.85\\
            65 & 379220.0 & \text{opt} & 0.0 & 0 & 379220.0 & 0.01 & 379220.0 & 0.01 & 379220.0 & 0 & 0.0\\
            66 & 164134.0 & \text{opt} & 2.0 & 10774 & 164134.0 & 2586.77 & \text{-} & 3600.61 & 164134.0 & 5171 & 4.63\\
            67 & 703224.0 & \text{opt} & 0.0 & 0 & 703224.0 & 0.02 & 703224.0 & 0.02 & 703224.0 & 0 & 0.0\\
            68 & 94517.5 & \text{opt} & 0.0 & 51 & 94517.5 & 0.02 & 94517.5 & 0.03 & 94517.5 & 52 & 0.42\\
            69 & 65259.0 & \text{inf} & 3.0 & 3410 & \text{-} & 3672.31 & \text{-} & 3600.29 & 65259.0 & 5138 & 4.73\\
            70 & -2612.74 & \text{opt} & 0.0 & 0 & -2612.74 & 0.0 & -2612.74 & 0.01 & -2612.74 & 0 & 0.0\\
            71 & 667875.0 & \text{inf} & 1.0 & 1478 & \text{-} & 3750.4 & \text{-} & 3600.43 & 667875.0 & 1458 & 2.83\\
            72 & -627719.0 & \text{opt} & 0.0 & 0 & \text{-} & 3719.42 & \text{-} & 3600.33 & -627719.0 & 84 & 0.8\\
            73 & 48011.4 & \text{inf} & 0.0 & 24 & \text{-} & 3603.42 & \text{-} & 3600.45 & 48011.4 & 352 & 0.97\\
            74 & 71518.2 & \text{opt} & 0.0 & 371 & \text{-} & 3610.12 & \text{-} & 3600.52 & 71518.2 & 678 & 1.24\\
            75 & 16119.0 & \text{opt} & 0.0 & 0 & 16119.0 & 4.08 & 16119.0 & 1.42 & 16119.0 & 239 & 0.17\\
            76 & 4976.64 & \text{opt} & 0.0 & 0 & 4976.64 & 0.07 & 4976.64 & 0.02 & 4976.64 & 0 & 0.01\\
            77 & -3343.57 & \text{opt} & 0.0 & 186 & -3343.57 & 779.71 & -3343.57 & 1525.54 & -3343.57 & 222 & 0.17\\
            78 & -3209.93 & \text{opt} & 0.0 & 0 & -3209.93 & 0.0 & -3209.93 & 0.01 & -3209.93 & 0 & 0.0\\
            79 & 64146.1 & \text{opt} & 0.0 & 0 & \text{-} & 3742.06 & \text{-} & 3600.29 & 64146.1 & 3982 & 2.82\\
            80 & 15155.4 & \text{inf} & 0.0 & 0 & 15155.4 & 147.59 & 15155.4 & 58.44 & 15155.4 & 287 & 0.26\\
            81 & 84971.2 & \text{opt} & 0.0 & 1000 & \text{-} & 3711.35 & \text{-} & 3600.43 & 84971.2 & 252119 & 3148.39\\
            82 & 33343.4 & \text{opt} & 0.0 & 0 & \text{-} & 3661.97 & \text{-} & 3600.54 & 33343.4 & 227 & 0.31\\
            83 & 775.58 & \text{opt} & 0.0 & 25 & 775.58 & 0.01 & 775.58 & 0.18 & 775.58 & 142 & 0.18\\
    \end{tabular}
    \end{adjustbox}
    \caption{Performance of model $\model{}$ solved with CPLEX, the RCSPP algorithm, HA* and BORWin on instances 43 to 83 of set A}
    \label{tab:res_borwin_a2}
\end{table}\begin{table}[!htb]
    \centering
    \begin{adjustbox}{max width=1\textwidth, center=\textwidth}
    \begin{tabular}{r||S[table-format=7.2]|S[table-format=1.2]|S[table-format=4.2]|r||S[table-format=7.2]|S[table-format=4.2]||S[table-format=7.2]|S[table-format=4.2]||S[table-format=7.2]|r|S[table-format=4.2]}
        & \multicolumn{4}{c}{CPLEX} & \multicolumn{2}{c||}{RCSPP}& \multicolumn{2}{c||}{HA*} & \multicolumn{3}{c}{BORWin}\\
        \text{instance} & \text{value} & \text{gap} & \text{time} & \text{\#nodes} & \text{value} & \text{time} & \text{value} & \text{time} & \text{value} & \text{\#iter} & \text{time} \\
            1 & -45097.8 & 0.24 & 3599.0 & 5159794 & \text{-} & 3728.16 & \text{-} & 3600.2 & -45097.8 & 230460 & 2318.99\\
            2 & 30816.8 & \text{opt} & 462.0 & 535149 & 30816.8 & 2.66 & \text{-} & 3600.33 & 30816.8 & 4385 & 2.73\\
            3 & -12364.3 & 0.87 & 3599.0 & 6511928 & \text{-} & 3657.08 & \text{-} & 3600.29 & -12364.3 & 185759 & 1325.15\\
            4 & -4771.67 & \text{opt} & 88.0 & 51047 & \text{-} & 3635.03 & \text{-} & 3600.44 & -4771.67 & 13212 & 73.31\\
            5 & 100037.0 & \text{opt} & 0.0 & 23 & \text{-} & 3627.04 & \text{-} & 3600.55 & 100037.0 & 6 & 0.79\\
            6 & -8330.37 & \text{opt} & 182.0 & 95280 & \text{-} & 3625.38 & \text{-} & 3600.56 & -8330.37 & 21368 & 58.23\\
            7 & -25268.1 & \text{opt} & 0.0 & 494 & \text{-} & 3641.08 & \text{-} & 3600.57 & -25268.1 & 12275 & 25.31\\
            8 & -15274.7 & \text{inf} & 0.0 & 105 & \text{-} & 3786.47 & -15274.7 & 198.96 & -15274.7 & 1943 & 3.93\\
            9 & -134802.0 & 0.01 & 3599.0 & 797938 & \text{-} & 3676.52 & \text{-} & 3600.33 & -134885.0 & 5155 & 3601.74\\
            10 & -495393.0 & \text{opt} & 907.0 & 119597 & \text{-} & 3628.56 & \text{-} & 3600.21 & -495393.0 & 3428 & 7.31\\
            11 & 22498.4 & \text{opt} & 1.0 & 1079 & 22498.4 & 50.23 & \text{-} & 3600.42 & 22498.4 & 185 & 2.41\\
            12 & -15704.4 & \text{opt} & 91.0 & 19017 & \text{-} & 3744.61 & \text{-} & 3600.31 & -15704.4 & 26 & 5.27\\
            13 & -362953.0 & 0.01 & 3599.0 & 749964 & \text{-} & 4030.62 & \text{-} & 3600.26 & -362953.0 & 106576 & 3601.59\\
            14 & 20036.7 & \text{opt} & 812.0 & 680819 & 20036.7 & 28.34 & \text{-} & 3600.42 & 20036.7 & 4734 & 1.65\\
            15 & 3060.83 & \text{opt} & 0.0 & 38 & 3060.83 & 0.02 & 3060.83 & 1648.99 & 3060.83 & 1214 & 0.47\\
            16 & 9849.72 & \text{opt} & 12.0 & 34012 & 9849.72 & 0.07 & \text{-} & 3600.39 & 9849.72 & 2722 & 1.3\\
            17 & 8633.68 & \text{opt} & 2.0 & 9440 & 8633.68 & 0.17 & \text{-} & 3600.51 & 8633.68 & 10660 & 6.11\\
            18 & 10167.2 & 0.4 & 3599.0 & 1124478 & 10167.2 & 0.07 & \text{-} & 3600.51 & 10167.2 & 7908 & 3.63\\
            19 & 8617.41 & \text{opt} & 1.0 & 5022 & 8617.41 & 0.04 & 8617.41 & 3360.88 & 8617.41 & 1965 & 0.91\\
            20 & 8781.47 & \text{opt} & 13.0 & 34045 & 8781.47 & 0.64 & \text{-} & 3600.52 & 8781.47 & 4018 & 1.74\\
            21 & 0.0 & \text{opt} & 0.0 & 0 & 0.0 & 0.0 & 0.0 & 0.01 & 0.0 & 1 & 0.01\\
            22 & 21996.5 & \text{opt} & 0.0 & 0 & 21996.5 & 0.05 & 21996.5 & 0.01 & 21996.5 & 0 & 0.0\\
            23 & 95051.4 & 0.63 & 3599.0 & 833133 & 95051.4 & 30.72 & \text{-} & 3600.4 & 95051.4 & 10440 & 31.82\\
            24 & 18974.6 & \text{opt} & 0.0 & 0 & 18974.6 & 0.02 & 18974.6 & 0.01 & 18974.6 & 0 & 0.0\\
            25 & 196839.0 & \text{opt} & 0.0 & 0 & 196839.0 & 1.83 & 196839.0 & 758.65 & 196839.0 & 1879 & 1.94\\
            26 & 21459.5 & \text{opt} & 0.0 & 0 & 21459.5 & 0.04 & 21459.5 & 0.01 & 21459.5 & 2 & 0.19\\
            27 & 95900.6 & \text{opt} & 0.0 & 0 & 95900.6 & 0.03 & 95900.6 & 0.01 & 95900.6 & 0 & 0.0\\
            28 & 102450.0 & \text{opt} & 0.0 & 0 & 102450.0 & 0.03 & 102450.0 & 0.01 & 102450.0 & 0 & 0.0\\
            29 & 62292.3 & \text{inf} & 2.0 & 6388 & \text{-} & 3626.06 & \text{-} & 3600.41 & 62292.3 & 55078 & 93.38\\
            30 & -5397.97 & \text{opt} & 0.0 & 0 & -5397.97 & 0.01 & -5397.97 & 0.01 & -5397.97 & 0 & 0.0\\
            31 & 99.34 & \text{opt} & 0.0 & 0 & 99.34 & 0.01 & 99.34 & 0.02 & 99.34 & 97 & 0.3\\
            32 & 24792.1 & \text{opt} & 0.0 & 0 & 24792.1 & 0.01 & 24792.1 & 0.01 & 24792.1 & 0 & 0.0\\
            33 & 3079.3 & \text{opt} & 0.0 & 0 & 3079.3 & 0.06 & 3079.3 & 0.01 & 3079.3 & 0 & 0.0\\
            34 & 59431.0 & \text{opt} & 0.0 & 0 & 59431.0 & 153.67 & \text{-} & 3600.5 & 59431.0 & 3689 & 2.18\\
            35 & 100832.0 & \text{opt} & 0.0 & 0 & 100832.0 & 0.03 & 100832.0 & 0.01 & 100832.0 & 0 & 0.0\\
            36 & 27052.8 & \text{opt} & 0.0 & 0 & 27052.8 & 0.44 & \text{-} & 3600.66 & 27052.8 & 427 & 0.48\\
            37 & 406009.0 & \text{sub} & 1.0 & 670 & \text{-} & 3662.9 & \text{-} & 3600.22 & 406023.0 & 134167 & 3601.37\\
            38 & -4316.46 & \text{opt} & 0.0 & 0 & -4316.46 & 85.35 & \text{-} & 3600.57 & -4316.46 & 5271 & 3.26\\
            39 & -42097.1 & \text{opt} & 4.0 & 17021 & \text{-} & 3695.39 & \text{-} & 3600.55 & -42097.1 & 3445 & 2.55\\
            40 & 15916.1 & \text{opt} & 0.0 & 124 & 15916.1 & 79.01 & \text{-} & 3600.5 & 15916.1 & 2444 & 2.87\\
            41 & -3187.16 & \text{opt} & 0.0 & 0 & -3187.16 & 0.0 & -3187.16 & 0.01 & -3187.16 & 97 & 0.15\\
            42 & 4529.62 & \text{opt} & 0.0 & 0 & 4529.62 & 0.04 & 4529.62 & 0.01 & 4529.62 & 0 & 0.0\\
    \end{tabular}
    \end{adjustbox}
    \caption{Performance of model $\model{}$ solved with CPLEX, the RCSPP algorithm, HA* and BORWin on  instances 1 to 42 of set B}
    \label{tab:res_borwin_b1}
\end{table}\begin{table}[!htb]
    \centering
    \begin{adjustbox}{max width=1\textwidth, center=\textwidth}
    \begin{tabular}{r||S[table-format=7.2]|S[table-format=1.2]|S[table-format=4.2]|r||S[table-format=7.2]|S[table-format=4.2]||S[table-format=7.2]|S[table-format=4.2]||S[table-format=7.2]|r|S[table-format=4.2]}
        & \multicolumn{4}{c}{CPLEX} & \multicolumn{2}{c||}{RCSPP}& \multicolumn{2}{c||}{HA*} & \multicolumn{3}{c}{BORWin}\\
        \text{instance} & \text{value} & \text{gap} & \text{time} & \text{\#nodes} & \text{value} & \text{time} & \text{value} & \text{time} & \text{value} & \text{\#iter} & \text{time} \\
            43 & 269050.0 & \text{opt} & 0.0 & 0 & 269050.0 & 0.04 & 269050.0 & 0.01 & 269050.0 & 0 & 0.0\\
            44 & 917201.0 & 0.06 & 3599.0 & 8558238 & \text{-} & 3866.17 & \text{-} & 3600.4 & 917035.0 & 190950 & 3601.38\\
            45 & 282424.0 & \text{sub} & 3.0 & 5610 & \text{-} & 3627.51 & \text{-} & 3600.49 & 282425.0 & 3858 & 2.95\\
            46 & 52834.5 & \text{opt} & 0.0 & 0 & 52834.5 & 0.06 & 52834.5 & 0.03 & 52834.5 & 0 & 0.01\\
            47 & 459095.0 & \text{opt} & 0.0 & 0 & 459095.0 & 0.05 & 459095.0 & 0.01 & 459095.0 & 0 & 0.0\\
            48 & 452277.0 & 0.04 & 3599.0 & 5319419 & \text{-} & 3625.98 & \text{-} & 3600.47 & 452318.0 & 123282 & 2254.47\\
            49 & 42341.8 & \text{opt} & 0.0 & 0 & 42341.8 & 0.06 & 42341.8 & 0.02 & 42341.8 & 0 & 0.0\\
            50 & -5494.62 & 4.65 & 3599.0 & 1298359 & \text{-} & 3671.95 & \text{-} & 3600.47 & -5494.62 & 91615 & 527.29\\
            51 & 40192.6 & \text{opt} & 0.0 & 0 & 40192.6 & 0.01 & 40192.6 & 0.02 & 40192.6 & 0 & 0.0\\
            52 & 42283.9 & \text{opt} & 10.0 & 34235 & \text{-} & 3641.65 & \text{-} & 3600.54 & 42283.9 & 6194 & 4.61\\
            53 & 27442.4 & \text{opt} & 0.0 & 0 & 27442.4 & 0.66 & 27442.4 & 0.1 & 27442.4 & 31 & 0.04\\
            54 & 3798460.0 & \text{sub} & 0.0 & 0 & 3798740.0 & 20.38 & \text{-} & 3600.47 & 3798740.0 & 31065 & 22.19\\
            55 & -1957.72 & \text{opt} & 2322.0 & 898717 & \text{-} & 3626.03 & \text{-} & 3600.54 & -1957.72 & 7688 & 19.39\\
            56 & 28471.9 & \text{opt} & 0.0 & 0 & 28471.9 & 0.01 & 28471.9 & 0.02 & 28471.9 & 0 & 0.01\\
            57 & 3922410.0 & \text{sub} & 0.0 & 352 & \text{-} & 3831.41 & \text{-} & 3600.29 & 3922560.0 & 76710 & 395.57\\
            58 & 716594.0 & \text{opt} & 2930.0 & 844090 & \text{-} & 3692.37 & \text{-} & 3600.62 & 714666.0 & 193649 & 3601.47\\
            59 & 17577.3 & \text{opt} & 2.0 & 1768 & 17577.3 & 5.06 & 17577.3 & 52.03 & 17577.3 & 514 & 1.34\\
            60 & 25720.8 & \text{opt} & 0.0 & 0 & 25720.8 & 0.04 & 25720.8 & 0.01 & 25720.8 & 0 & 0.0\\
            61 & -34801.1 & \text{inf} & 9.0 & 40790 & \text{-} & 3680.74 & \text{-} & 3600.56 & -34801.1 & 14291 & 23.78\\
            62 & 10607.8 & \text{opt} & 16.0 & 13431 & 10607.8 & 290.81 & 10607.8 & 2028.97 & 10607.8 & 3113 & 8.36\\
            63 & 935608.0 & \text{opt} & 0.0 & 0 & 935608.0 & 0.05 & 935608.0 & 0.02 & 935608.0 & 0 & 0.0\\
            64 & 95093.1 & \text{inf} & 0.0 & 0 & \text{-} & 3632.43 & \text{-} & 3600.61 & 95093.1 & 286 & 0.53\\
            65 & 379220.0 & \text{opt} & 0.0 & 0 & 379220.0 & 0.01 & 379220.0 & 0.01 & 379220.0 & 0 & 0.0\\
            66 & 138473.0 & 0.37 & 3599.0 & 651399 & 138473.0 & 325.18 & \text{-} & 3600.7 & 138473.0 & 48951 & 381.27\\
            67 & 703224.0 & \text{opt} & 0.0 & 0 & 703224.0 & 0.02 & 703224.0 & 0.02 & 703224.0 & 0 & 0.0\\
            68 & 94516.4 & \text{opt} & 0.0 & 57 & 94516.4 & 0.02 & 94516.4 & 0.02 & 94516.4 & 32 & 0.62\\
            69 & 66301.3 & \text{inf} & 7.0 & 4272 & \text{-} & 3659.77 & \text{-} & 3600.33 & 66301.3 & 2960 & 3.43\\
            70 & -2612.74 & \text{opt} & 0.0 & 0 & -2612.74 & 0.0 & -2612.74 & 0.01 & -2612.74 & 0 & 0.0\\
            71 & 618927.0 & \text{opt} & 16.0 & 16554 & \text{-} & 3665.96 & \text{-} & 3600.52 & 618927.0 & 24436 & 111.57\\
            72 & -842932.0 & \text{sub} & 2.0 & 588 & \text{-} & 3608.11 & \text{-} & 3600.32 & -842891.0 & 2503 & 4.41\\
            73 & 30603.6 & \text{opt} & 27.0 & 39539 & \text{-} & 3652.74 & \text{-} & 3600.37 & 30603.6 & 23572 & 48.14\\
            74 & 68651.0 & 0.1 & 3599.0 & 5690332 & \text{-} & 3663.35 & \text{-} & 3600.73 & 68651.0 & 8209 & 15.63\\
            75 & 15102.9 & \text{opt} & 0.0 & 0 & 15102.9 & 4.68 & \text{-} & 3600.54 & 15102.9 & 4086 & 3.91\\
            76 & 4976.64 & \text{opt} & 0.0 & 0 & 4976.64 & 0.06 & 4976.64 & 0.02 & 4976.64 & 0 & 0.0\\
            77 & -6975.34 & \text{opt} & 1.0 & 1676 & -6975.34 & 57.44 & \text{-} & 3600.63 & -6975.34 & 1202 & 1.35\\
            78 & -3209.93 & \text{opt} & 0.0 & 0 & -3209.93 & 0.0 & -3209.93 & 0.01 & -3209.93 & 0 & 0.0\\
            79 & 57148.6 & \text{opt} & 524.0 & 446548 & \text{-} & 3671.69 & \text{-} & 3600.48 & 57148.6 & 176537 & 1347.51\\
            80 & 14650.7 & \text{opt} & 0.0 & 0 & 14650.7 & 31.81 & \text{-} & 3600.64 & 14650.7 & 6600 & 5.74\\
            81 & 69986.4 & \text{opt} & 32.0 & 73328 & \text{-} & 3705.61 & \text{-} & 3600.38 & 69986.4 & 147543 & 1148.09\\
            82 & 31375.2 & \text{sub} & 0.0 & 0 & \text{-} & 3609.26 & \text{-} & 3600.6 & 31375.9 & 95735 & 476.4\\
            83 & 164.95 & \text{opt} & 0.0 & 600 & 164.95 & 0.01 & 164.95 & 1.59 & 164.95 & 492 & 0.47\\
    \end{tabular}
    \end{adjustbox}
    \caption{Performance of model $\model{}$ solved with CPLEX, the RCSPP algorithm, HA* and BORWin on  instances 43 to 83 of set B}
    \label{tab:res_borwin_b2}
\end{table}

\end{document}